\documentclass[a4paper,11pt]{article}
\usepackage[latin1]{inputenc}
\usepackage{amsmath}
\usepackage{amsfonts}
\usepackage{amssymb}
\usepackage{amscd}
\usepackage{dsfont}
\usepackage{mathrsfs}
\usepackage{graphicx}
\usepackage{epstopdf}
\usepackage{setspace}
\usepackage{fancyhdr}
\usepackage{amsthm}
\usepackage{empheq}
\usepackage{cases}
\usepackage[all]{xy}

\newtheorem*{theorem*}{Theorem}
\newtheorem*{proposition*}{Proposition}
\newtheorem{theorem}{Theorem}[section]
\newtheorem{definition}{Definition}[section]

\newtheorem{proposition}[theorem]{Proposition}

\newtheorem{remark}[theorem]{Remark}

\newcommand{\erre}{\mathds{R}}

\newcommand{\essen}{\mathds{S}^n}
\newcommand{\erren}{\mathds{R}^n}

\newcommand{\esse}{\mathds{S}}
\newcommand{\E}{\mathds{E}}

\newcommand{\mob}{\mathrm{M\ddot{o}b}}
\newcommand{\mab}{\mathfrak{m\ddot{o}b}}
\newcommand{\di}{\mathrm{d}} 

\newcommand{\ra}{\rightarrow}

\newcommand{\set}[1]{{\left\{#1\right\}}}               
\newcommand{\pa}[1]{{\left(#1\right)}}                  
\newcommand{\sq}[1]{{\left[#1\right]}}                  
\newcommand{\abs}[1]{{\left|#1\right|}}                 
\newcommand{\pair}[1]{\left\langle#1\right\rangle}      

\newcommand{\eps}{\varepsilon}                           

\newcommand{\metric}{\pair{\;,}}                          

\renewcommand{\hat}[1]{\widehat{#1}}
\renewcommand{\tilde}[1]{\widetilde{#1}}

\begin{document}

\author{Marco Magliaro \and Luciano Mari \and Marco Rigoli}
\title{\textbf{On the geometry of curves and conformal geodesics in the M\"obius
space}}
\date{}
\maketitle
\scriptsize \begin{center} Dipartimento di Matematica,
Universit\`a
degli studi di Milano,\\
Via Saldini 50, I-20133 Milano (Italy)\\
E-mail addresses: \\
marco.magliaro@unimi.it, \ luciano.mari@unimi.it, \
marco.rigoli@unimi.it
\end{center}

\begin{abstract}
This paper deals with the study of some properties of immersed
curves in the conformal sphere $\mathds{Q}_n$, viewed as a
homogeneous space under the action of the M\"obius group. After an overview on
general well-known facts, we briefly focus on the links between
Euclidean and conformal curvatures, in the spirit of F. Klein's
Erlangen program. The core of the paper is the study of conformal
geodesics, defined as the critical points of the conformal
arclength functional. After writing down their Euler-Lagrange
equations for any $n$, we prove an interesting codimension
reduction, namely that every conformal geodesic in $\mathds{Q}_n$
lies, in fact, in a totally umbilical $4$-sphere
$\mathds{Q}_4$. We then extend and complete the work in
\cite{musso} by solving the Euler-Lagrange equations for the
curvatures and by providing an explicit expression even for those
conformal geodesics not included in any conformal
$3$-sphere.
\end{abstract}
 \maketitle

\normalsize

\section{Introduction}
The investigation of the conformal properties of submanifolds of
the unit $n$-dimensional sphere is a well-developed field in
differential geometry. Its deep links range, for example, from the
classical theory of curves and surfaces in $\erre^3$ to the theory
of integrable systems, general relativity and so on. Among the
various approaches that have been used in the study of the
subject, Cartan's method of the moving frame (see \cite{jensen},
\cite{sharpe} and the original books of E. Cartan,
\cite{cartan_gr}, \cite{cartan_rep}) stands out for its
usefulness, depth and generality. In particular, the theory of
homogeneous spaces, which encompasses the conformal geometry of
the sphere, gives a vast field where this technique applies at
best. In this paper, we use Cartan's method to deal with the
geometry of immersed curves. In particular, our main concern is to
complete the characterization of conformal geodesics $f: I\subset
\erre \ra \mathds{Q}_n$, begun with the work of E. Musso in
\cite{musso} in the case of $\mathds{Q}_3$. Such curves arise as
the stationary points of the integral of a conformally invariant
$1$-form, $\di s$, called the conformal arclength. Roughly
speaking, $\di s$ is the conformal analogue of the differential of
the Euclidean arclength, and is linked to it through some
classical formulas in \cite{liebmann}, \cite{sul}, \cite{takasu},
\cite{vessiot}.
\\
Using Griffiths' formalism (see \cite{gri}), in \cite{musso} the
author wrote the Euler-Lagrange equations for $n=3$,
obtaining the following system of ODEs for the conformal
curvatures $\mu_1,\mu_2\in C^\infty(I)$:
\begin{equation}\label{eultre1}
\left\{ \begin{array}{ll}
\dot{\mu}_1+ 3 \mu_2\dot{\mu}_2=0 \\[0.2cm]
\ddot{\mu}_2 = \mu_2^3 + 2\mu_1\mu_2,
\end{array}\right.
\end{equation}
where the dot denotes the derivative with respect to the arclength
parameter $s$; $\{\mu_1,\mu_2\}$ constitutes a complete set of
conformal invariants which characterize $f$ up to a conformal
motion of $\mathds{Q}_3$. He then solved \eqref{eultre1} by means
of elliptic functions, and found the explicit form of every
conformal geodesic. His method leads to a lengthy computation as
the dimension $n$ grows, thus the need of a different approach for
a generic $n$. In this paper (Section \ref{sezionearco}), we
obtain a simple form for the Euler-Lagrange equations in any
dimension without the aid of Griffiths' formalism. This expression
leads to the following, unexpected, degeneracy result (Theorem
\ref{teoreduz}): \
\begin{theorem*}
Every conformal geodesic $f:I\subset \erre \ra \mathds{Q}_n$ lies in some totally umbilical $4$-sphere $\mathds{Q}_4
\subset \mathds{Q}_n$.
\end{theorem*}
This allows us to limit ourselves to the study of conformal
geodesics in $\mathds{Q}_4$. After characterizing degenerate
cases, we concentrate on generic conformal geodesics in the
$4$-sphere, obtaining the Euler-Lagrange equations
\begin{equation}\label{eulquattro1}
\left\{ \begin{array}{ll}
\dot{\mu}_1+ 3 \mu_2\dot{\mu}_2=0, \qquad \mu_2>0,\\[0.2cm]
\ddot{\mu}_2 = \mu_2^3 + 2\mu_1\mu_2 + \mu_2\mu_3^2, \qquad \mu_3>0,\\[0.2cm]
2\dot{\mu}_2\mu_3 + \mu_2 \dot{\mu}_3 =0.
\end{array}\right.
\end{equation}
relating the conformal curvatures $\{\mu_1,\mu_2,\mu_3\}$ of a conformal geodesic. As it is apparent, system \eqref{eulquattro1} exhibits a clear similarity
with \eqref{eultre1} and can be integrated by using elliptic functions (Section \ref{sez_equazionimoto}). This
suggests that an explicit integration of the equations of motion
can be achieved even for $n=4$. As a matter of fact, we obtain an explicit expression for $f$ up to a conformal motion, see the discussion after Proposition \ref{soluzionecurvatura}.
Surprisingly, this process reveals even easier than the corresponding one for $n=3$, since a unique case has to be analyzed, differently from
\cite{musso} where the author has to
deal with three distinct cases.\\
\par
We need to introduce some background material for our purposes.
Although part of it is quite standard, for the convenience of the
reader and to keep the paper basically self-contained and the
notation coherent, we have decided to begin with a brief account
on the basic results in the conformal geometry of curves. The
first two sections therefore deal with the homogeneous space
structure of the $n$-sphere (Section \ref{sez_essen}) and the
geometry of immersed curves (Section \ref{sez_frenet}). Our
approach follows closely that of \cite{ScSulanke} and \cite{sul}
except for a slightly different notation. Among classical and
modern references, we mainly recommend \cite{hert} and
\cite{sharpe}: they both present a complete, self-contained
account on the conformal group and two different, very nice
geometrical proofs of Liouville's theorem. An elementary
introduction to the subject is also given in \cite{bear}, and
standard further references are \cite{vranceanu},
\cite{eisenhart}, \cite{berger}. For subtle topological details,
we suggest to the reader \cite{palais}. In Section
\ref{sez_euclidea} we discuss the links between Euclidean and
conformal invariants. Our approach leads to a (slight) improvement
of a result in \cite{romerofuster} by showing that their set of
conformal invariants indeed coincides with the one presented here,
see Proposition \ref{rome}. Another application of the circle of
ideas of Section \ref{sez_euclidea} is a natural proof of the
conformal invariance of the so-called total twist of a closed
curve in $\mathds{Q}_3$, that is, the quantity
$$
\mathrm{Tw}(f)= \frac{1}{2\pi} \int_I \tau\di s_e \qquad (\mathrm{mod} \ \mathds{Z}),
$$
where $\tau$ is the Euclidean torsion of $f$ and $s_e$ is the
Euclidean arclength. This result has been proved in \cite{CSWebb},
but it seems that their elegant argument, although not far from
our approach, is based on some sort of ``magical'' identity
involving a globally defined angle, whose appearing seems to us
not completely justified.

\section{Construction of the conformal
sphere $\mathds{Q}_n$: a short review}\label{sez_essen}

Consider $\essen$ and $\erren$ with their standard metrics of
constant curvatures, and let $\sigma: \essen\backslash \{N\} \ra
\erren$ be the stereographic projection, where $N=(1,0,\ldots,
0)\in \erre^{n+1}$ is the north pole. It is well known that
$\sigma$ is a conformal diffeomorphism. If $n\ge 3$, by
Liouville's theorem (\cite{dubnovfom}, pp.138-141; \cite{hert},
pp.52-53, \cite{sharpe}, pp. 289-290), every conformal
diffeomorphism of $\essen$ is of the form $\sigma^{-1}\circ g
\circ \sigma$, where $g$ is a composition of Euclidean
similarities of $\erren$ with possibly the inversion $
\erren\backslash \{0\} \ni x\mapsto x/|x|^2$. The assertion holds
even for $n=2$, although a proof of this fact relies, for
instance, on standard compact Riemann surfaces theory since
Liouville's theorem is false for $\mathds{C}$. We observe that the
conformal group $\mathrm{Conf}(\esse^2)$ can be also identified
with the fractional linear transformation of $\mathds{C}$, either
holomorphic or anti-holomorphic. From now on, we let $n\ge 2$ and
we fix the index convention $1\le A,B,C\le n$. We denote with
$\mathds{Q}_n$ the Darboux hyperquadric
$$
\mathds{Q}_n = \set{ (x^0:x^A:x^{n+1}) \ | \ \sum_A (x^A)^2 - 2x^0x^{n+1}=0}
\subset \mathds{P}^{n+1}(\erre).
$$
The Dirac-Weyl embedding $\chi:\erren\ra\mathds{Q}_n$ is defined
by
$$
\chi : x\longmapsto \pa{1 : x : \frac 12 |x|^2}
$$
and it extends to a diffeomorphism $\chi\circ \sigma : \essen \ra
\mathds{Q}_n$ by setting $\chi\circ\sigma (N) =(0:0:1)$. The
advantage of such a representation for the sphere is that every
conformal diffeomorphism of $\essen$ acts as a linear
transformation on the homogeneous coordinates of $\mathds{Q}_n$,
so that $\mathrm{Conf}(\essen)$ can be viewed as the projectivized
of the linear subgroup of $GL(n+2)$ preserving the quadratic form
which
defines the Darboux hyperquadric. \\
Endow $\erre^{n+2}$ with the Lorentzian metric $\metric$
represented, with respect to the standard basis
$\{\eta_0,\eta_A,\eta_{n+1}\}$, by the matrix
$$
S=\left(\begin{array}{ccc}
0 & 0 & -1 \\
0 & I_n & 0 \\
-1 & 0 & 0
\end{array}\right),
$$
and let $L^+$ be the positive light cone, that is, $L^+=\{v= \,
^t(v^0,v^A,v^{n+1}) \in \erre^{n+2} : \, ^tvSv=0, \ v^0+ v^{n+1} >
0 \}$. Note that $L^+$ projectivizes to $\mathds{Q}_n$ and that
$\eta_0,\eta_{n+1} \in L^+$. Moreover, there is a bijection
between $\mathrm{Conf}(\essen)$ and the Lorentz group of $\metric$
preserving the positive light cone (usually called the
orthochronous Lorentz group). This gives a Lie group structure to
the conformal group $\mathrm{Conf}(\essen)$, which can be proved
to be unique when the action of $\mathrm{Conf}(\essen)$ on
$\essen$ is required to be smooth (see \cite{palais}, pp. 95-98).
In particular, the identity component of the Lorentz group is
called the \textbf{M\"obius group}, $\mob(n)$, and coincides with
the subgroup of the orientation preserving elements of
$\mathrm{Conf}(\essen)$. The transitivity of the action of
$\mob(n)$ on the $n$-sphere gives $\mathds{Q}_n$ a homogeneous
space structure, allowing us to identify it with the space of left
cosets $\mob(n)/\mob(n)_0$, where $\mob(n)_0$ is the isotropy
subgroup of $[\eta_0]\in \mathds{Q}_n$:
\begin{equation}\label{G0}
\mob(n)_0 = \set{\left(
\begin{array}{ccc}
r^{-1} & ^txA & \frac{1}{2}r|x|^2\\
0 &  A & rx \\
0 & 0 & r
\end{array}
\right) \left|
\begin{array}{l}
r >0, \ x \in \erren, \\[0.2cm]
A \in SO(n)
 \end{array}
 \right. }.
\end{equation}
It follows that the principal bundle projection $\pi: \mob(n) \ra
\mathds{Q}_n$ associates to a matrix $G=(g_0|g_A|g_{n+1})$ the
point $[G\eta_0]= [g_0]\in \mathds{Q}_n$. From now on, we shall
use the Einstein  summation convention. Let $\mab(n)$ denote the
Lie algebra of $\mob(n)$; the Maurer-Cartan form $\Phi$ of
$\mob(n)$ is the $\mab(n)$-valued $1$-form
$$
\Phi = \left(
\begin{array}{ccc}
\Phi^0_0 & \Phi^0_B & 0 \\[0.2cm]
\Phi^A_0 & \Phi^A_B & \Phi^A_{n+1} \\[0.2cm]
0 & \Phi^{n+1}_B & \Phi^{n+1}_{n+1}
\end{array}
\right),
$$
with the symmetry relations
$$
\Phi^{n+1}_{n+1} = -\Phi^0_0, \quad \Phi^A_B=-\Phi^B_A,\quad
\Phi^A_{n+1} = \Phi^0_A, \quad  \Phi^{n+1}_B = \Phi^B_0
$$
and satisfying the structure equation $\di \Phi +
\Phi\wedge\Phi=0$, which component-wise reads
\begin{equation}\label{eqstruttura}
\left\{ \begin{array}{rcl} \di \Phi^0_0 & = & -\Phi^0_A \wedge
\Phi^A_0; \\[0.2cm]
\di \Phi^A_0 & = & -\Phi^A_0 \wedge
\Phi^0_0 - \Phi^A_B \wedge \Phi^B_0; \\[0.2cm]
\di \Phi^0_A & = & -\Phi^0_0 \wedge
\Phi^0_A - \Phi^0_B \wedge \Phi^B_A; \\[0.2cm]
\di \Phi^A_B & = & -\Phi^A_0 \wedge \Phi^0_B -\Phi^A_C\wedge \Phi^C_B - \Phi^0_A \wedge \Phi^B_0.
\end{array}\right.
\end{equation}
Through a local section $s: U\subset \mathds{Q}_n\ra \mob(n)$,
$\Phi$ pulls back to a flat Cartan connection $\phi= s^*\Phi=
s^{-1}\di s$. In particular, the set $\{\phi^A_0\}$ gives a local
basis for the cotangent bundle of $\mathds{Q}_n$. Under a change
of section $\tilde{s}=sK$, where $K: U\subset \mathds{Q}_n\ra
\mob(n)_0$, the change of gauge becomes
$$
\tilde{\phi} = \tilde{s}^{-1}\di \tilde{s} = K^{-1}\phi K + K^{-1}\di K.
$$
By the expression of $\mob(n)_0$ in \eqref{G0}, we have in
particular
\begin{equation}\label{cambioprimo}
(\tilde{\phi}^A_0) = r^{-1}\,^t\!A(\phi^A_0),
\end{equation}
where $(\phi^A_0)$ stands for the column vector whose $A$-th
component is $\phi^A_0$. It follows that
$$
\tilde{\phi}^A_0 \otimes \tilde{\phi}^A_0 = r^{-2} \phi^A_0\otimes
\phi^A_0, \qquad \tilde{\phi}^1_0\wedge \ldots \wedge
\tilde{\phi}^n_0 = r^{-n} \phi^1_0\wedge \ldots \wedge \phi^n_0,
$$
which implies that
$$
\Big\{\big(U,\phi^A_0\otimes\phi^A_0\big)\ : \ U\subset \mathds{Q}_n \text{
domain of a local section } s:U\ra \mob(n) \Big\}
$$
defines a conformal structure on $\mathds{Q}_n$, that is, a
collection of locally defined metrics varying conformally on the
intersection of their domains of definition, together with an
orientation (locally defined by $\phi^1_0\wedge \ldots \wedge
\phi^n_0$), both preserved by $\mob(n)$. It is easy to prove that,
with this conformal structure, $\chi\circ\sigma : \essen \ra
\mathds{Q}_n$ is a conformal diffeomorphism. This gives sense to
the whole construction.
\section{The Frenet-Serret equations for curves in
$\mathds{Q}_n$}\label{sez_frenet}

Let $I \subset \erre $ be an open interval and let $f : I
\rightarrow \mathds{Q}_n$ be an immersion. We give an outline of
the frame reduction procedure and deduce the Frenet-Serret
equation for the curve $f$. This is a standard procedure, see for
example \cite{sul}, \cite{musso}. Henceforth we adopt the
following index conventions:
$$
1 \le A,B,C, \ldots \le n \qquad , \qquad 2 \le \alpha,\beta,
\ldots \le n.
$$
Let $e:U\subset M\ra\mob(n)$ be a zeroth order frame along $f$,
namely a smooth map such that $\pi\circ e=f_{|U}$, and set
$\phi=e^*\Phi$. If $\tilde e$ is another zeroth order frame along
$f$, then $\tilde e = eK$, where $K$ is a $\mob(n)_0$-valued
smooth map. It follows that
\begin{equation}\label{cambiogauge}
\widetilde{\phi}= \widetilde{e}^*\Phi = \widetilde{e}^{-1}\di
\widetilde{e} = K^{-1}\phi K + K^{-1}\di K.
\end{equation}
From \eqref{cambioprimo} and since $f$ is an immersion, for a
fixed point $p \in I$, it is always possible to consider a frame
$e$ such that
\begin{equation}\label{primacond}
\phi^\alpha_0 := e^*\Phi^\alpha_0 =0
\end{equation}
at $p$. The isotropy subgroup of such frames is
\begin{equation}\label{G1}
\mob(n)_1 = \set{ \left(
\begin{array}{cccc}
r^{-1} & x & ^tyB & \frac{1}{2}r(x^2+|y|^2)\\
0 &  1 & 0 & rx \\
0 & 0 & B & ry \\
0 & 0 & 0 & r
\end{array}
\right) \left|
\begin{array}{l}
r >0, \ x \in \erre \\
y \in \erre^{n-1},\\
B \in SO(n-1)
 \end{array}
 \right.}
\end{equation}
and since it is independent of $p$, by the standard theory of
frame reduction (see \cite{sulsvec}, \cite{sulwint},
\cite{sharpe}) smooth zeroth order frames can be chosen, which
satisfy condition \eqref{primacond} in a suitable
neighbourhood of $p$. Such frames will be called \textbf{first order frames}.\\
If $e$ and $\tilde e$ are first order frames, they are related by
$\tilde e=eK$, where now $K$ takes values in $\mob(n)_1$. From
\eqref{cambiogauge} we get
\begin{equation}\label{primuz}
\widetilde{\phi}^1_0= r^{-1} \phi^1_0, \qquad \widetilde{\phi}^\alpha_1= B^\beta_\alpha\pa{\phi^\beta_1-y^\beta\phi^1_0},
\end{equation}
thus the form $\phi^1_0$ determines a conformal structure on $I$.
If we set $\phi^\alpha_1 = h^\alpha\phi^1_0$, where $h^\alpha$ are
smooth functions  locally defined on $I$, then
\begin{equation}\label{cambhalfamono}
\tilde h^\alpha =rB^\beta_\alpha\pa{h^\beta-y^\beta}.
\end{equation}
At any point $p\in I$, we can therefore consider a first order
frame such that $h^\alpha$, hence $\phi^\alpha_1$, vanishes. Since
the isotropy subgroup preserving such frames is
\begin{equation}\label{G2}
\mob(n)_2 = \set{ \left(
\begin{array}{cccc}
r^{-1} & x & 0 &\frac{1}{2}rx^2\\
0 &  1 & 0 & rx \\
0 & 0 & B & 0 \\
0 & 0 & 0 & r
\end{array}
\right) \left|
\begin{array}{l}
r >0, \ x \in \erre \\
B \in SO(n-1)
\end{array}
\right.},
\end{equation}
independent of $p$, we can locally choose a frame satisfying $h^\alpha=0$. Such frames will be called \textbf{second order frames} or \textbf{Darboux frames} along $f$.\\
Under a change of Darboux frames $\tilde e=eK$, for a $\mob(n)_2$-valued $K$, denoting
by $e_a$ the $a$-th column of $e$, we have
\begin{equation}\label{cambframesec}
\left\{\begin{array}{rcl}
\tilde\phi^\alpha_\beta &=& B^\gamma_\alpha\phi^\gamma_\delta B^\delta_\beta + B^\gamma_\alpha\di B^\gamma_\beta \\[0.2cm]
\tilde e_\alpha &=& B^\beta_\alpha e_\beta. \end{array}\right.
\end{equation}
Thus we can define a vector bundle, denoted by $N$ and called the
\textbf{normal bundle}, by declaring $\set{e_\alpha}$ a local
basis. When $n=2$, the normal bundle is simply the span of $e_2$.
$N$ is endowed with a global inner product, defined by requiring
$\{e_\alpha\}$ to be orthonormal, and a compatible connection
$\overline{\nabla}$ by setting
\begin{equation*}
\overline{\nabla} e_\alpha = \phi^\beta_\alpha \otimes e_\beta.
\end{equation*}
We set
\begin{equation}\label{palpha}
\phi^0_\alpha = p^\alpha \phi^1_0,
\end{equation}
for some smooth local functions $p^\alpha$. Then, under a change of Darboux frames,
\begin{equation}\label{cambpalpha}
\tilde p^\alpha = r^2 B^\beta_\alpha p^\beta,
\end{equation}
which, together with \eqref{primuz} implies that the form
\begin{equation}\label{nond}
\sqrt[4]{\sum_\alpha (\tilde p^\alpha)^2} \phi^1_0
\end{equation}
is globally defined. Note that it may be only of class $C^{0,1/2}$
locally around the points where it vanishes. Since $I$ is an
interval, there exists a function $s: I \rightarrow\erre$ of class
$C^{1,1/2}(I)$ such that
\begin{equation}\label{lungharco}
\sqrt[4]{\sum_\alpha (\tilde p^\alpha)^2} \phi^1_0 = \di s.
\end{equation}
Observe that $s$ is defined up to a constant. Moreover, up to
changing the sign, it is non-decreasing and smooth, strictly
increasing on every connected subinterval of $I$ where
$\sum_\alpha(p^\alpha)^2>0$.
\begin{definition}\label{arcoconforme}
Every function $s:I\rightarrow \erre$ such that \eqref{lungharco}
holds is called a \emph{\textbf{conformal arclength}}.
\end{definition}
Driven by geometrical considerations, we make the non-degeneracy
assumption that $\di s$ never vanishes on $I$.
\begin{definition}\label{defnondegen}
Given an immersed curve $f: I \rightarrow \mathds{Q}_n$, a point
$q\in I$ is called \emph{\textbf{$1$-generic}} if
\begin{equation}\label{nondegen}
\sum_\alpha (p^\alpha)^2 \neq 0
\end{equation}
at $q$. The immersion $f$ is said to be
\emph{\textbf{$1$-generic}} if every point of $I$ is $1$-generic,
and it is said to be \emph{\textbf{totally $1$-degenerate}} if
$\sum_\alpha (p^\alpha)^2 \equiv 0$ on $I$. A totally
$1$-degenerate point is called a \emph{\textbf{vertex}} of $f$.
\end{definition}
It shall be noted that, when $n\ge 3$, $1$-generic curves are an
open, dense subset (with respect to the $C^\infty$ topology) of
the space of smooth curves, either closed or not. This is a
consequence of strong transversality, see \cite{arnold}. On the
contrary, global topological obstructions appear when the ambient
space has dimension $2$. For instance, the well-known four-vertex
theorem ensures that every smooth, simple closed curve in
$\mob(2)$ has at least four vertexes. For an account on this
result in its various forms one can consult \cite{CSWebb},
\cite{osserman}, \cite{pinkall}.\par
Given a $1$-generic curve $f$ and a point $p$, by
\eqref{cambpalpha} we can always choose a Darboux frame with the
property that $p^2=1$, $p^3=\ldots=p^n=0$, that is,
\begin{equation}\label{terzor}
\phi^0_2 = \phi^1_0=\di s, \qquad \phi^0_3 = \phi^0_4 = \ldots =
\phi^0_n = 0.
\end{equation}
The isotropy subgroup for such frames is clearly
\begin{equation}\label{G3}
\mob(n)_3 = \set{ \left(
\begin{array}{ccccc}
1 & x & 0 & 0 &\frac{1}{2}x^2\\
0 &  1 & 0 & 0 & x \\
0 & 0 & 1 & 0 & 0 \\
0 & 0 & 0 & C & 0 \\
0 & 0 & 0 & 0 & 1
\end{array}\right) \left|
\begin{array}{ccc}
   x \in \erre  \\
  C \in SO(n-2)
\end{array}
\right.}.
\end{equation}
Note that, if $n=2$, the rows and columns containing $C$ do not
appear at all and, if $n=3$, $C$ reduces to the real number $1$.\\
Since this subgroup does not depend on $p$, we can define a
\textbf{third order frame} along a $1$-generic $f$ as a second
order frame satisfying \eqref{terzor}. Now, proceeding with the
reduction, we set
\begin{equation}\label{definqtilde}
\phi^0_0 = q^2 \phi^1_0, \qquad \phi^b_2 = q^b \phi^1_0 \quad
\text{ for } \ b \in \{3,\ldots,n\}, \ n\ge 3,
\end{equation}
for some locally defined smooth functions $q^\alpha$. Under a change of third order frame we have
\begin{equation}\label{qtilde}
\tilde q^2 = q^2-x, \qquad \tilde q^b=C^d_b q^d,
\end{equation}
therefore at every point $p\in I$ we can choose a third order frame such that $q^2=0$, hence $\phi^0_0=0$. The isotropy subgroup preserving such frames is
\begin{equation}\label{G3spec}
\mob(n)_{3\mathrm{spec}} = \set{ \left(
\begin{array}{ccc}
I_3 & 0 & 0 \\
0 & C & 0 \\
0 & 0 & 1
\end{array}
\right) |\ C \in
SO(n-2)},
\end{equation}
so the reduction can be performed smoothly around every $1$-generic point. We define a
\textbf{special third order frame} along a $1$-generic curve $f$ as a third order frame such that
$$
\phi^0_0 = 0.
$$
Now the form $\phi^0_1$ is independent of the chosen special third
order frame so, writing
\begin{equation}\label{deflambda1}
\phi^0_1 = \mu_1\phi^1_0= \mu_1 \di s,
\end{equation}
defines a conformal invariant $\mu_1 \in C^\infty(I)$, which may
change sign on $I$. We point
out that the sign of $\mu_1$ cannot be reversed by changing the (oriented) special third order frame.\\
If $n=2$ or $3$, \eqref{G3spec} reduces to the identity, and this
completes the procedure. A special third order frame along a
$1$-generic curve $f$ will be called a \textbf{Frenet frame}.
Moreover, when $n=3$, $q^3$ is a third order invariant of the
generic curve $f$, that we shall denote by
$\mu_2$. In this case, also the sign of $\mu_2$ may vary on $I$.\\
Assume now $n\ge 3$. The structure of the isotropy subgroup
\eqref{G3spec} (actually, even that of \eqref{G3}) implies that
\begin{equation}\label{cambE}
\tilde{e}_2 = e_2, \qquad \tilde e_b = C_b^c e_c \qquad b,c \in \{3,\ldots,n\},
\end{equation}
hence $N$ splits as the Whitney sum $<e_2>\oplus \, \Theta$, where
$\Theta$ is locally spanned by $\{e_b\}$ and is endowed with a
natural connection, defined by setting
\begin{equation}\label{connexion}
\nabla e_b = \phi^c_b \otimes e_c,
\end{equation}
which is compatible with the Riemannian structure induced by $N$. We denote with $|\cdot |$ the induced norm on $\Theta$.\\
By \eqref{cambE} and \eqref{qtilde}, the smooth section
\begin{equation}\label{defX}
X = q^be_b \in \Gamma(\Theta)
\end{equation}
is independent of the third order frame considered, hence globally defined.\\
As in Definition \ref{defnondegen}, to be able to proceed we need a second non-degeneracy condition:
\begin{definition}\label{defnondegen2}
Given an immersed curve $f: I \rightarrow \mathds{Q}_n$, $n\ge 3$,
a point $q\in I$ is called \emph{\textbf{$2$-generic}} if it is
$1$-generic and, for a third order frame,
\begin{equation}\label{nondegen2}
\phi^c_{2} \neq 0 \quad \text{at } q \text{ for at least one }
c\in \{3,\ldots,n\},
\end{equation}
that is, $X(q)\neq 0$. The immersion $f$ is
\emph{\textbf{$2$-generic}} if it is $2$-generic at every point,
and it is \emph{\textbf{totally $2$-degenerate}} if it is
$1$-generic and $\phi^c_{2} \equiv 0$ on $I$ for every
$c\in\{3,\ldots,n\}$, that is, $X\equiv 0$.
\end{definition}
When $n\ge 4$, for every $2$-generic point $p\in I$ we can choose
a special third order frame $e$ such that $X$ points in the
direction of $e_3$, that is,
\begin{equation}\label{4red}
q^3 = \mu_2> 0, \quad q^b = 0
\quad \text{for } \ b \in \{4,\ldots,n\}.
\end{equation}
The isotropy subgroup preserving such frames has the structure
\begin{equation}\label{G4}
\mob(n)_4 = \set{ \left(
\begin{array}{ccc}
I_4 & 0 & 0 \\
0 & C_1 & 0 \\
0 & 0 & 1
\end{array}
\right) | \ C_1 \in SO(n-3)}.
\end{equation}
This shows that, for $2$-generic curves, we can proceed with the reduction as usual, and if $n=4$ this is the last step. A special third order frame along a $2$-generic curve $f$ satisfying
$$
\phi^b_2 = 0 \quad \text{ for } \ b
\in \{4,\ldots,n\}.
$$
will be called a \textbf{fourth order frame}. If $n=4$ this frame
will be called a \textbf{Frenet frame} since it gives
the last reduction.\\
Summarizing what we have got so far:
\begin{itemize}
\item[-] if $n=2$, in a Frenet frame we can write the
    pull-back $\phi$ of the Maurer-Cartan form as
\begin{equation}\label{cartanformfrenet}
\phi = \left(
\begin{array}{cccc}
0 & \mu_1 \di s & \di s & 0 \\
\di s & 0 & 0 & \mu_1 \di s \\
0 & 0 & 0 & \di s \\
0 & \di s & 0 & 0
\end{array}
\right);
\end{equation}
\item[-] if $n=3$, in a Frenet frame we can write the
    pull-back $\phi$ of the Maurer-Cartan form as
\begin{equation}\label{cartanformfrenet}
\phi = \left(
\begin{array}{ccccc}
0 & \mu_1 \di s & \di s & 0 & 0 \\
\di s & 0 & 0 & 0 & \mu_1 \di s \\
0 & 0 & 0 & -\mu_2 \di s & \di s \\
0 & 0 & \mu_2\di s & 0 & 0 \\
0 & \di s & 0 & 0 & 0
\end{array}
\right);
\end{equation}
\item[-] If $n=4$, the form $\phi^4_3$ is invariant, and we can therefore write
$$
\phi^4_3 = \mu_3 \phi^1_0 = \mu_3 \di s
$$
for some smooth, possibly changing sign function $\mu_3$ on
$I$. The Maurer-Cartan form $\phi$ has the expression
\begin{equation}\label{cartanformfrenet4}
\phi = \left(
\begin{array}{cccccc}
0 & \mu_1 \di s & \di s & 0 &  0 & 0 \\
\di s & 0 & 0 & 0 & 0 & \mu_1 \di s \\
0 & 0 & 0 & -\mu_2 \di s & 0 & \di s \\
0 & 0 & \mu_2\di s & 0 & -\mu_3 \di s & 0 \\
0 & 0 & 0 & \mu_3 \di s & 0 & 0 \\
0 & \di s & 0 & 0 & 0 & 0
\end{array}
\right).
\end{equation}
\end{itemize}
The previous expressions, together with the definition of $\phi$, that is $\di e = e\phi$ give rise to the Frenet formulae
\begin{equation}\label{frenetformulae}
\begin{array}{cc}
n=3 & n=4 \\[0.2cm]
\left\{
\begin{array}{rcl}
\di e_0 & = & \di s \, e_1 \\
\di e_1 & = & \mu_1 \di s \, e_0 + \di s \, e_4\\
\di e_2 & = & \di s \, e_0 +\mu_2\di s \, e_3\\
\di e_3 & = & -\mu_2 \di s \, e_2 \\
\di e_4 & = & \mu_1 \di s \, e_1 + \di s\, e_2
\end{array}
\right. \quad &
\left\{
\begin{array}{rcl}
\di e_0 & = & \di s \, e_1 \\
\di e_1 & = & \mu_1 \di s \, e_0 + \di s \, e_5\\
\di e_2 & = & \di s \, e_0 +\mu_2\di s \, e_3\\
\di e_3 & = & -\mu_2 \di s \, e_2 + \mu_3 \di s \, e_4 \\
\di e_4 & = & -\mu_3 \di s \, e_3 \\
\di e_5 & = & \mu_1 \di s \, e_1 + \di s\, e_2
\end{array}\right.
\end{array}
\end{equation}
The general reduction steps for $n\ge5$ can now be carried on
inductively.\\
Let $4\le k\le n-1$. For a $(k-2)$-generic curve,
writing $\phi^c_{k-1} = q^c_{k-1}\phi^1_0$, $c\ge k$, and keeping
in mind the isotropy subgroup \eqref{G4}, the vector field
$$
X_{(k-1)} = q^c_{k-1}e_c, \qquad k\le c\le n
$$
Is globally defined and independent of the chosen $k$-th order
frame (in this notation, it is convenient for the reader to rename
$X$ in \eqref{defX} as $X_{(2)}$). Moreover, by construction
$X_{(k-1)}$ is orthogonal to the span of $e_3,\ldots,e_{k-1}$.
\begin{definition}\label{defnondegenk}
Given an immersed curve $f: I \rightarrow \mathds{Q}_n$ and an
integer $4\le k<n$, a point $q\in I$ is called
\emph{\textbf{$(k-1)$-generic}} if it is $(k-2)$-generic and, for
$k^\mathrm{th}$-order frames
\begin{equation}\label{nondegenk}
\phi^c_{k-1} \neq 0 \quad \text{ at } q \text{ for at least one } c\in \set{k,\ldots,n},
\end{equation}
or, equivalently, $X_{(k-1)}(q)\neq 0$. The immersion $f$ is
$(k-1)$-generic if it is $(k-1)$-generic at every point, and is
\emph{\textbf{totally $(k-1)$-degenerate}} if it is
$(k-2)$-generic and $X_{(k-1)}\equiv 0$ on $I$.
\end{definition}
For $(k-1)$-generic curves, we can choose a $k^\mathrm{th}$ order
frame such that $X_{(k-1)}$ points in the direction of $e_k$ at
$p\in I$, that is,
\begin{equation}\label{polarizz}
\phi^k_{k-1} = \mu_{k-1} \phi^1_0, \quad \phi^b_{k-1} = 0 \quad \text{ for } \ b \in \set{k+1,\ldots,n}
\end{equation}
and $\mu_{k-1} = |X_{(k-1)}|> 0$. The isotropy subgroup of the
reduction has the form
\begin{equation}\label{Gk}
\mob(n)_{k+1} = \set{ \left(
\begin{array}{ccc}
I_{k+1} & 0 & 0 \\
0 & C_{k-2} & 0 \\
0 & 0 & 1
\end{array}
\right) | \ C_{k-2} \in SO(n-k) },
\end{equation}
and since it does not depend on the point $p\in I$,
we can smoothly define a $(k+1)^{\text{th}}$-order frame along a $(k-1)$-generic curve
$f$ as a $k^{\text{th}}$-order frame along $f$ such that
\begin{equation*}
  \phi^b_{k-1} = 0 \ \text{ for } \ b \in \set{k+1,\ldots,n}.
\end{equation*}
%
%
%
Finally, for $k=n-1$, we have constructed, on an $(n-2)$-generic curve,
an $n^{\text{th}}$-order frame and, by \eqref{Gk}, the reduction is complete.\\
The $n^\mathrm{th}$ order frame so constructed is called a \textbf{Frenet frame}.\\
The only form left untreated is the now globally defined $\phi^n_{n-1}$. We set
\begin{equation*}
  \phi^n_{n-1}=\mu_{n-1}\phi^1_0, \qquad X_{(n-1)}= \mu_{n-1}e_n
\end{equation*}
for some $\mu_{n-1}\in C^\infty(I)$, not necessarily positive.
According to the above definitions, we say that $f$ is
\textbf{$(n-1)$-generic} if $\mu_{n-1}\neq 0$ at every point of
$I$, and \textbf{totally $(n-1)$-degenerate} if $\mu_{n-1}\equiv
0$.\\
For those who prefer working with Koszul formalism, the above
reduction procedure on the bundle $\Theta$ can be rephrased and
summarized as follows: at each step $k$, we identify a vector
$X_{(k-1)}\in \Gamma(\Theta)$ independent of the chosen frame;
then, if $|X_{(k-1)}|> 0$, locally we set $e_{k}=
X_{(k-1)}/|X_{(k-1)}|$ and we covariantly differentiate $e_{k}$
along $f$ with respect to the connection $\nabla$ on $\Theta$.
$X_{(k)}$ is defined as the component of $\nabla e_{k}/\di s$
orthogonal to the span of $\{e_3,\ldots, e_k\}$. The
non-degeneracy condition is equivalent to the non vanishing of
this component. With all the genericity assumptions, we can
proceed up until we provide an orthonormal basis of $\Theta$.\par
The Frenet-Serret equations $\di e = e \phi$ for an
$(n-2)$-generic $f$ read
\begin{equation}\label{frenetserretn}
\left\{
\begin{array}{lcl}
\di e_0 & = & \di s \, e_1 \\
\di e_1 & = & \mu_1 \di s \, e_0 + \di s \, e_{n+1}\\
\di e_2 & = & \di s \, e_0 +\mu_2\di s \, e_3\\
\di e_k & = & -\mu_{k-1} \di s \, e_{k-1} + \mu_k \di s \, e_{k+1} \qquad k \in \{3,\ldots, n-1\}.\\
\di e_n & = & -\mu_{n-1} \di s \, e_{n-1} \\
\di e_{n+1} & = & \mu_1 \di s \, e_1 + \di s\, e_2
\end{array}
\right.
\end{equation}
The following characterization of degeneracy can be proved.
Although this result already appears in \cite{sul}, we present
here a slightly different proof of group-theoretical nature.
\begin{proposition}\label{casototalmdegenere}
Let $I\subset \erre$ be an interval and $f:I\rightarrow \mathds{Q}_n$, $n\ge 3$, be an immersion.
Then $f$ is totally $1$-degenerate if and only if there exists a conformal circle $\mathds{Q}_1\subset
\mathds{Q}_n$ such that $f(I)\subset Q_1$.\\
Moreover, for every $k\in \{2,\ldots,n-1\}$ if $f$ is a
$(k-1)$-generic curve, then $f$ is totally $k$-degenerate if and
only if there exists a conformal $k$-sphere $\mathds{Q}_k\subset
\mathds{Q}_n$ such that $f(I) \subset \mathds{Q}_k$.
\end{proposition}
\begin{proof}
We assume $k\ge 2$, the other case being analogous. Fix the index
convention
$$
\eta,\nu \in \{2,\ldots, k\}, \qquad b,c \in \{k+1,\ldots, n\}.
$$
Consider on $\mob(n)$ the ideal $\mathcal{I}$ generated by the
forms $ \Phi^b_0, \Phi^b_1, \Phi^b_\eta, \Phi^0_b$. Using the
structure equations \eqref{eqstruttura} and the symmetries of
$\mab(n)$, $\mathcal{I}$ can be proved to be a differential ideal,
that is, $\di \mathcal{I}\subset \mathcal{I}$. The distribution
$\Delta$ defined by $\mathcal{I}$ is therefore integrable.
Moreover, at the identity, $\Delta$ is given by
\begin{equation}\label{algebralieT}
\Delta_I=\set{
\left(
\begin{array}{cccc}
  a & {}^tx & 0 & 0 \\
  y & D  & 0 & x \\
  0 & 0 & E & 0 \\
  0 & {}^ty & 0 & -a
\end{array}
\right)\left|
\begin{array}{c}
  D\in \mathfrak{o}(k)   \\
  E\in \mathfrak{o}(n-k)   \\
  x,y\in\erre^k  \\
  a\in\erre
\end{array}
\right.}
\end{equation}
and it is obtained at any other point by left translation because of the left invariance of $\Phi$. In particular, its maximal integral submanifold passing through the identity is the following subgroup of $\mob(n)$:
\begin{equation*}
\begin{array}{c}
T=\set{
\left(
\begin{array}{cccc}
  a & {}^tz & 0 & b \\
  x & A  & 0 & y \\
  0 & 0 & B & 0 \\
  c & {}^tw & 0 & d
\end{array}
\right)\left| \left(
\begin{array}{ccc}
  a & {}^tz  & b \\
  x & A   & Y \\
  c & {}^tw & d
\end{array}
\right)\right. \in\mob(k),\ B\in SO(n-k)
 }\simeq\\[1cm]
 \simeq\mob(k)\times SO(n-k).
\end{array}
\end{equation*}
We denote by $\tau,\xi$ the projections
$$
\tau: T\ra \mob(k), \qquad \xi:T\ra SO(n-k).
$$
For every other leaf $\Sigma$ of the distribution, there exists a
constant element $G\in \mob(n)$ such that $L_G(\Sigma) = T$.
Moreover, the intersection $T\cap\mob(n)_0$ is isomorphic to
$\mob(k)_0 \times SO(n-k)$. This implies that the following
diagram commutes, where $\pi_n,\pi_k$ are the projections that
define $\mathds{Q}_n$ and $\mathds{Q}_k$ and $[\tau\times \xi]$ is
the naturally defined quotient map:
\begin{equation}\label{diagramma}
\xymatrix{
  \Sigma\ar[r]^{L_G}_\sim \ar[d]^{\pi_n} & T  \ar[d]^{\pi_n} \ar[r]^(.25){\tau\times \xi}_(.25)\sim & \mob(k)\times SO(n-k) \ar[d]^{\pi_k\times 0 } \\
\pi_n(\Sigma) \ar[r]^{G\cdot}_\sim & \pi_n(T) \ar[r]^{[\tau\times
\xi]}_\sim & \mathds{Q}_k }
\end{equation}
From the diagram we deduce that each leaf $\pi_n(\Sigma)$ of the
distribution ${\pi_n}_*\Delta$ on $\mathds{Q}_n$ is conformally
equivalent to a conformal $k$-sphere. Let now $f$ be totally
$k$-degenerate. Then, in a frame $e$ of suitable order,
\begin{equation}\label{buono}
0=\phi^b_0=\phi^b_1=\phi^b_\eta=\phi^0_b,
\end{equation}
hence $e_*TI$ is a subset of the distribution $\Delta$. It follows that $e(I)\subset \Sigma$ for some leaf
$\Sigma$, hence $f(I)=(\pi_n\circ e)(I)\subset \pi_n(\Sigma)$.\\
Since $\tau\times \xi$ is a Lie group isomorphism, it is easy to
check that the conformal invariants (as well as every other
conformal property) of $f$ seen as a curve in $\mathds{Q}_k$ are
the same as those of $f$ seen as a totally $k$-degenerate curve in
$\mathds{Q}_n$. Up to a conformal motion $G$, we can thus assume
$e(I)\subset T$. Since $T$ leaves the Lorentzian $(k+2)$-subspace
$<\eta_0,\eta_1,\ldots,\eta_k,\eta_{n+1}>$ invariant and is
invertible, it follows that $<e_0,e_1,\ldots, e_k,e_{n+1}>\equiv
<\eta_0,\eta_1,\ldots, \eta_k,\eta_{n+1}>$ and projects to
$\mathds{Q}_k$.\\
Conversely, if $f(I)$ is a subset of some conformal $k$-sphere
(that is, the projectivization of some coset of $T$), up to a
conformal motion we can assume that $f(I)\subset \pi(T)$. Let
$p\in I$. If $\sigma:\mathds{Q}_k\ra T$ is a local section around
$f(p)$, then $e=\sigma\circ f$ is a local frame around $p$, whose
form $\phi=e^{-1}\di e$ trivially satisfies \eqref{buono}. This
shows that $f$ is totally $k$-degenerate.
\end{proof}
As Proposition \ref{casototalmdegenere} suggests, when $f$ is
$(k-1)$-generic we might expect there to be a unique $k$-sphere
whose contact with the curve is at least of order $k$; moreover,
the above proof hints that, for $k^{\text{th}}$ order frames, this
$k$-sphere be the one that comes from the linear subspace of
$\erre^{n+2}$ spanned by$\set{e_0,e_1,\ldots,e_k,e_{n+1}}$.
However, classically the osculating sphere of order $k$ is defined
as the projectivization of the following subspace:
$$
V_k(t) = < e_0(t),\dot{e}_0(t),\ddot{e}_0(t),\ldots, e_0^{(k+1)}(t)> \quad \forall \ 1\le k\le n,
$$
where the dot stands for usual derivation with respect to the
parameter $t$. From this definition, it is not even clear whether
this space has the right dimension or not. However, from $\di e =
e\phi$ for a $k^{\text{th}}$ order frame, it is easy to prove that
\begin{equation}\label{oscsphere}
V_k(t)\equiv <e_0(t), e_1(t), \ldots, e_k(t), e_{n+1}(t)>.
\end{equation}
The projectivization of the intersection of this Minkowski
subspace with the light cone will be called the \textbf{osculating
sphere of order $k$} at $f(t)$.\\
We recall here the Cartan-Darboux existence and uniqueness theorem
for curves, stating that an immersed curve is completely
characterized, up to a conformal diffeomorphism, by its conformal
invariants $\{\di s, \mu_1,\ldots,\mu_{n-1}\}$.
\begin{theorem}[\cite{sulsvec} Th.3.2, 4.2, 4.3; \cite{sharpe}, p.119]\label{rigidita}
Let $I\subset \erre$ be a compact interval, $s: I\rightarrow
\erre$ be a strictly increasing smooth function and $\mu_1,\dots,
\mu_{n-1}$, $n\ge 2$, be smooth functions on $I$ such that
\begin{equation}\label{nondegn}
\mu_k > 0 \quad \text{ at every point of } \ I, \qquad k \in\{2,\ldots,n-2\}.
\end{equation}
Then, up to a conformal motion of $\mathds{Q}_n$, there exists a
unique $(n-2)$-generic immersed curve $f:I\rightarrow
\mathds{Q}_n$ having $s$ (up to an additive constant) as the
conformal arclength, and $\mu_1,\ldots,\mu_{n-1}$ as given
conformal invariants.
\end{theorem}
\section{Euclidean and conformal invariants of
curves}\label{sez_euclidea}
Let $F: I \subset \erren$ be a smooth curve. In this section, we
describe the links between Euclidean and conformal properties of
$F$. Through the Dirac-Weyl chart, we can view $\erren$ as an open
set of the conformal sphere $\mathds{Q}_n$. If we represent the
group of rigid motions $\mathds{E}(n)$ as a subgroup of $\mob(n)$
in the following way:
$$
J \ : \ (x,A) \in \mathds{E}(n) \longmapsto \left(\begin{array}{ccc} 1 & 0 & 0 \\
x & A & 0
\\ \frac 12 |x|^2 & ^txA & 1 \end{array} \right) \in \mob(n),
$$
$x\in\erren$, $A\in SO(n)$, then the following diagram commutes:
$$
\xymatrix{
& \mathds{E}(n) \ar[r]^(.4)J \ar[d]^{\pi_\mathds{E}} & \mob(n) \ar[d]^\pi \\
I \subset \erre \ar[r]^(.5)F & \erren \ar[r]^\chi & \mathds{Q}_n
}.
$$
where $\pi_\mathds{E}$ is the projection onto the first column. We
denote $f = \chi \circ F$. To every frame $E: I \ra \mathds{E}(n)$
along $F$, $E=(F,E_1,\ldots, E_n)$, we can associate $e= J \circ
E: I \ra \mob(n)$ along $f$. Moreover, up to the inclusion of Lie
algebras $J_{*,I_n}: \mathfrak{e}(n) \ra \mab(n)$, the
Maurer-Cartan form $\Psi$ of $\mathds{E}(n)$ is the pull-back of
$\Phi$ through $J$, so $\phi= E^*\Psi \equiv e^*\Phi$, which makes
the construction consistent. From now on, we identify
$\mathds{E}(n)$ with its image through $J$, the homogeneous space
$\pi_{\mathds{E}}: \mathds{E}(n) \ra \erren$ with $\pi:
J(\mathds{E}(n)) \ra \chi(\erren)$, $F$ with $f=\chi\circ F$ and
$E$ with $e=J\circ E$. We define $\mathds{E}(n)_0$ to be the
isotropy subgroup of $[\eta_0]=\chi(0)$ in $\mathds{E}(n)$. In
this framework, a frame reduction can be carried on and gives the
(Euclidean) Frenet frame along $f: I \ra \erren$. We very briefly
sketch the procedure, which is analogous to that performed in the
previous section. The first reduction gives a first order frame,
characterized by the condition $\phi^\alpha_0=0$, and the subgroup
of $\mathds{E}(n)_0$ preserving first order frames is
\begin{equation}\label{E1}
\mathds{E}(n)_1 = \set{ \left(
\begin{array}{ccc}
I_2 & 0 & 0\\
0 & B & 0 \\
0 & 0 & 1
\end{array}
\right) \left| B \in SO(n-1).\right. }
\end{equation}
First order frames are often called \textbf{(Euclidean) Darboux
frames}, and identify a well defined normal bundle $N_\mathds{E}$,
$N_{\mathds{E}}= < e_\beta >$, endowed with an inner product $( \,
, \, )$ (and induced norm $\|\cdot\|$) by requiring $\{e_\beta\}$
to be orthonormal, together with a compatible connection
$\nabla^e$ given by $\nabla^e e_\beta = \phi_\beta^\alpha \otimes
e_\alpha$. When $n=2$, $\mathds{E}(n)_1$ reduces to the identity
matrix, and
$N_{\mathds{E}}=<e_2>$.\\
With respect to first order frames, $\phi^1_0$ is globally defined
and never vanishing. This gives rise (up to an additive constant)
to the \textbf{Euclidean arclength} $s_e : I \ra \erre$ such that
$\phi^1_0= \di s_e$. Writing $\phi^\alpha_1 = h^\alpha \di s_e$,
the quantity $k = \sum_\alpha (h^\alpha)^2$ is independent of the
chosen Darboux frame, and is called the \textbf{curvature} of $f$.
It is easy to prove that $k\equiv 0$ if and only if the curve $f$
is a segment in $\erren$. Hereafter we make the non-degeneracy
assumption $k \neq 0$. Then, a further smooth reduction can be
made to have $h^2=k$ and, when $n\ge 3$, $h^b=0$ for $b\ge 3$. The
isotropy subgroup preserving such frames (\textbf{Euclidean second
order frames}) is
\begin{equation}\label{E1}
\mathds{E}(n)_2 = \set{ \left(
\begin{array}{cccc}
I_3 & 0 & 0\\
0 & B & 0 \\
0 & 0 & 1
\end{array}
\right) \left| B \in SO(n-2).\right. }
\end{equation}
Under a change $\widetilde{e}=eK$ of second order frames, $e_2$ is
invariant and defines the principal normal vector. Moreover, when
$n\ge 3$, setting
$$
v^b= \phi^b_2\left(\frac{\di}{\di s_e}\right), \qquad Y = v^be_b
\quad \text{for } 3\le b \le n
$$
it is immediate to show that $Y$ is globally defined and
independent of the chosen second order frame. Indeed, $Y= \nabla^e
e_2/\di s_e$. If $n=2$ we define $Y=0$ to avoid separating cases
in the next formulas. Define also
\begin{equation}\label{defZ}
Z = k' e_2 + kY = \frac{\nabla^e (ke_2)}{\di s_e}, \qquad
\text{where } k' = \frac{\di k}{\di s_e}.
\end{equation}
The subsequent steps can be carried on inductively. Observe that,
if $n=3$, $Y=\tau_2 e_3$ for some smooth, globally defined
function $\tau_2= \phi^3_2(\di/\di s_e)$, possibly changing sign;
$\tau_2$ (or just $\tau$, in this case) is called the
\textbf{torsion} of the curve. If $n\ge 4$, provided the
$2$-genericity condition $Y \neq 0$ is satisfied, a smooth
reduction can be made to have
$$
\phi^3_2 = \tau_2 \phi^1_0, \qquad \tau^b_2 = 0 \quad \text{for }
4 \le b \le n,
$$
namely, $Y=\tau_2e_3$ with $\tau_2>0$; the isotropy subgroup
preserving such frames is analogous to $\E(n)_2$, with $C\in
SO(n-3)$. Then we focus on the set of forms $\phi^b_3$, $b\ge 4$,
and so on. At every step $j$, the non-degeneracy condition reads
$\phi^b_{j-1}\neq 0$ for at least one index $b\ge j$, and it can
also be formulated as
$$
Y_{(j-1)} = \phi^b_{j-1}\left(\frac{\di}{\di s_e}\right)e_b \neq
0, \qquad j\le b\le n.
$$
The Euclidean invariants $\tau_3,\ldots, \tau_{n-1}$ are defined
by $Y_{(j)}=\tau_je_{j+1}$, $j\in\{3,\ldots,n-1\}$. In analogy
with Proposition \ref{casototalmdegenere}, $f$ is totally
$j$-degenerate ($Y_{(j)}\equiv 0$) if and only if $f(I)$ is a
subset of some affine $j$-subspace. Up to the identification
through the map $J$, the Euclidean Frenet-Serret equations $\di E
= E\phi$ for an $n$-generic curve $F$, $n\ge 3$, read
\begin{equation}
\left\{
\begin{array}{lcl}
\di F & = & \di s_e \, E_1 \\
\di E_1 & = & k \di s_e \, E_2 \\
\di E_2 & = & -k \di s_e \, E_1 + \tau_2 \di s_e \, E_3 \\
\di E_j & = & -\tau_{j-1} \di s_e \, E_{j-1} + \tau_j \di s_e \,
E_{j+1} \qquad 3 \le j \le n-1 \\
\di E_n & = & -\tau_{n-1}\di s_e\, E_{n-1}
\end{array}\right.
\end{equation}
In order to compare the Euclidean and conformal structures, we
observe that a Euclidean Darboux frame $e$ along a curve $f$ with
$k\neq 0$ is only a first order frame in the conformal sense. To
obtain a conformal second order frame ($\phi^\alpha_1=0$ for every
$\alpha$) we set $\widetilde{e} = eK$, where
$$
K = \left(\begin{array}{cccc} 1 & 0 & ^ty & \frac{1}{2}|y|^2 \\
0 & 1 & 0 & 0 \\
0 & 0 & I_{n-1} & y \\
0 & 0 & 0 & 1 \end{array} \right) \ \ ,  \qquad \begin{array}{l} y
= \left(\begin{array}{c}k \\0\end{array}\right) \in \erre^{n-1},
\quad n\ge 2,
\\[0.5cm]
y=k \quad \text{ if } \ n=2. \end{array}
$$
The forms change as follows:
\begin{equation}\label{primocamb}
\begin{array}{l}
\displaystyle \widetilde{\phi}^1_0 = \phi^1_0= \di s_e, \qquad
\widetilde{\phi}^\alpha_\beta = \phi^\alpha_\beta, \qquad
\widetilde{\phi}^0_1 = -\frac{k^2}{2} \phi^1_0,
\\[0.3cm] \widetilde{\phi}^0_2 = \di k = k'\phi^1_0, \qquad \widetilde{\phi}^0_b = k \phi^b_2 = kv^b\phi^1_0 \quad \text{for } \ b\ge 3.
\end{array}
\end{equation}
Then, from
$\widetilde{\phi}^0_\alpha=\widetilde{p}^\alpha\widetilde{\phi}^1_0$
we deduce $\widetilde{p}^2 = k'$, $\widetilde{p}^b = kv^b$. By
\eqref{defZ}, the curve is therefore $1$-generic in the conformal
sense if and only if
\begin{equation}\label{1gener}
\|Z\|^2 = (k')^2+ k^2\|Y\|^2 \neq 0 \qquad \text{on } I.
\end{equation}
Let $\hat{e}=\widetilde{e}K$, where
$$
K = \left(\begin{array}{cccc} r^{-1} & 0 & 0 & 0 \\
0 & 1 & 0 & 0 \\
0 & 0 & B & 0 \\
0 & 0 & 0 & r \end{array} \right) \in \mob(n)_2,
$$
for some $B\in SO(n-1)$, $r>0$. In order to obtain a third order
frame we must have $\hat{p}^2=1$, $\hat{p}^b=0$. From
\eqref{cambpalpha}, $1 = \sum_\alpha (\hat{p}^\alpha)^2 = r^4
\sum_\alpha (\widetilde{p}^\alpha)^2 = r^4\|Z\|^2$, hence
\begin{equation}\label{defr}
r = \frac{1}{\sqrt{\|Z\|}} = \left[ (k')^2+k^2\|Y\|^2
\right]^{-1/4}.
\end{equation}
Moreover, denoting by $B_\alpha$ the columns of $B$,
\eqref{cambpalpha} implies
\begin{equation}\label{coseB2}
B_2 = \left(\begin{array}{c} r^2k' \\ r^2kv^b \end{array}\right);
\end{equation}
roughly speaking, $B_2$ has the components of $Z/\|Z\|$. Using
also \eqref{primocamb}, the change of gauge gives:
\begin{equation}\label{secondocam}
\begin{array}{l}
\hat{\phi}^0_0 = -\di \log r, \qquad \hat{\phi}^0_1 =
r\widetilde{\phi}^0_1 =
-r\frac{k^2}{2} \di s_e, \\[0.2cm]
\hat{\phi}^1_0 = r^{-1}\widetilde{\phi}^1_0 = r^{-1} \di s_e,
\qquad (\hat{\phi}^\alpha_\beta) =
{}^t\!B(\widetilde{\phi}^\alpha_\beta)B + {}^t\!B\di B=
{}^t\!B(\phi^\alpha_\beta)B + {}^t\!B\di B.
\end{array}
\end{equation}
Since, for third order frames, $\hat{\phi}^1_0$ is the
differential of the conformal arclength, by \eqref{defr}
\begin{equation}\label{relazarchi}
\di s = \sqrt{\|Z\|} \, \di s_e = \left[ (k')^2+k^2\|Y\|^2
\right]^{1/4} \di s_e;
\end{equation}
If $n=2$, the above formula gives $\di s = \sqrt{|k'|} \di s_e$.
This relation shows that the classical definition of a vertex of a
closed plane curve (that is, a stationary point of the Euclidean
curvature) indeed reflects a conformal property, and coincides
with Definition \ref{defnondegen}; $\di s = \sqrt{|k'|} \di s_e$
first appeared in \cite{liebmann}, although the name of G. Pick is
also mentioned in \cite{blaschke} (see also
\cite{cairnssharpe} and \cite{bushmanova}).\\
If $n=3$, $\|Y\|^2 = \tau^2$, and \eqref{relazarchi} is again
classical and well known (see \cite{liebmann}, \cite{sul},
\cite{takasu}, \cite{vessiot}). When $n>3$ and $f$ is $2$-generic
in Euclidean sense, $Y=\tau_2e_3$ and $\di s = \left[
(k')^2+k^2\tau_2^2 \right]^{1/4} \di s_e$, see \cite{verbitskii}.\\
By \eqref{definqtilde}, we get
\begin{equation}\label{defq2}
\hat{q}^2 = -r\frac{\di \log r}{\di s_e} = -\frac{\di r}{\di s_e}.
\end{equation}
A special third order frame ($\overline{\phi}^0_0=0$) is thus
constructed by setting $\overline{e}=\hat{e}K$, with
$$
K = \left(\begin{array}{cccc} 1 & \hat{q}^2 & 0 & \hat{q}^2/2 \\
0 & 1 & 0 & \hat{q}^2 \\
0 & 0 & I_{n-1} & 0 \\
0 & 0 & 0 & 1 \end{array} \right) \in \mob(n)_3.
$$
This gives
\begin{equation}\label{terzocam}
(\overline{\phi}^\alpha_\beta)= (\hat{\phi}^\alpha_\beta), \qquad
\overline{\phi}^0_1 = \hat{\phi}^0_1 + \hat{q}^2\hat{\phi}^0_0 -
\frac 12 (\hat{q}^2)^2 \hat{\phi}^1_0 + \di (\hat{q}^2).
\end{equation}
For later use, if $n\ge 3$ we search for a simple expression for
the set of forms $(\phi^b_2)$. By \eqref{primocamb},
\eqref{secondocam} and \eqref{terzocam}, and using \eqref{coseB2},
\eqref{defZ} we get
\begin{equation}\label{phib2}
\begin{array}{lcl}
\overline{\phi}^b_2 & = & \big({}^t\!B(\phi^\alpha_\beta)B + {}^t\!B\di
B\big)^b_2 \\[0.2cm]
& = & B^2_b\phi^2_cB^c_2 + B^c_b \phi^c_2 B^2_2 + B^c_b\phi^c_d
B^d_2 + B^2_b \di B^2_2 + B^c_b \di B^c_2 \\[0.2cm]
& = & B_b^2\left[-r^2k\|Y\|^2 + (r^2k')'\right] \di s_e +
\\[0.2cm]
& & B^c_b\left[r^2k'v^c + r^2 k \phi^c_d\left(\frac{\di}{\di
s_e}\right)v^d + (r^2kv^c)'\right] \di s_e \\[0.3cm]
& = & \displaystyle (B_b^\alpha)
\left(\frac{\nabla^e(Z/\|Z\|)}{\di s_e}\right)^\alpha \di s_e
\end{array}
\end{equation}
%
%
%
%
From now on, we assume $f$ to be $(n-2)$-generic, so that a
complete frame reduction can be provided, giving the set of
invariants $\{\mu_1,\mu_2,\ldots, \mu_{n-1}\}$, possibly with
$\mu_{n-1}=0$ somewhere. It follows that $f$ is also generic from
the Euclidean point of view, whence we can also define the
Euclidean invariants $\{k,\tau_2,\ldots, \tau_{n-1}\}$. We aim to
relate such sets of invariants. By \eqref{deflambda1} and using
\eqref{secondocam} and \eqref{defq2} we obtain
\begin{equation}\label{cosemu1}
\mu_1 = \frac 12 \left(\frac{\di r}{\di s_e}\right)^2
-\frac{r^2k^2}{2} - r \frac{\di^2r}{\di s_e^2}, \qquad r=\left[(k')^2+k^2\tau_2^2\right]^{-1/4}.
\end{equation}
As we shall see, the general expression for $\mu_j$ is quite
complicated. In \cite{romerofuster} the authors have given a
simple expression for some conformal invariants called $\{k_j\}$,
$2\le j\le n-1$, in terms of the curvature radii $\{r_j\}$ of the
osculating spheres. These rise from a generalization of Coxeter's
inversive distance (see also \cite{sul3}). However, the authors do
not relate them to the set $\{\mu_j\}$. By the results of the
previous section, it is easy to show that the two sets indeed
coincide. Before doing this, it is convenient to recall their
procedure. A pair $S,\widetilde{S}$ of $j$-spheres in $\essen$ (or
$\erren$) are viewed in $\mathds{Q}_n$ as the projectivization of
suitable Lorentzian $(j+2)$-subspaces $V,\widetilde{V}$, which can
be explicitly computed. The orthogonal projections $\pi:
\erre^{n+2}\ra V$, $\widetilde{\pi}: \erre^{n+2}\ra \widetilde{V}$
with respect to the Lorentzian product $\metric$ are well defined,
self-adjoint operators. Setting
$$
I = \pi(\widetilde{V})\subset V, \quad \widetilde{I} = \widetilde{\pi}(V)\subset \widetilde{V},
\quad J= \ker(\widetilde{\pi}_{|V}) \subset V, \quad \widetilde{J}= \ker(\pi_{|\widetilde{V}})
\subset \widetilde{V},
$$
then $\pi: \widetilde{I}\ra I$, $\widetilde{\pi}: \widetilde{I}
\ra I$ are isomorphisms, $p_1=\pi\circ\widetilde{\pi}_{|I}$,
$p_2=\widetilde{\pi}\circ\pi_{|\widetilde{I}}$ are conjugate to
each other, hence they have the same characteristic polynomial.
Clearly, such polynomial is invariant under the action of the
M\"obius group. In particular, its trace $T_j$ is a conformal
invariant of the pair $(S,\widetilde{S})$. If $\set{w_0,\ldots,
w_{j+1}}$, (resp. $\set{\tilde{w}_0,\ldots,\tilde{w}_{j+1}}$) are
orthonormal bases which diagonalize the Minkowski inner product
$\metric$ and $\langle w_0,w_0\rangle=-1$ (resp. $\langle
\widetilde{w}_0,\widetilde{w}_0\rangle=-1$), $T_j$ is given by
$$
T_j(S,\widetilde{S})= \langle \widetilde{w}_0,w_0\rangle^2 - \sum_{i=1}^{j+1} \langle \widetilde{w}_0,w_i\rangle^2
-\sum_{i=1}^{j+1} \langle w_0,\widetilde{w}_i\rangle^2 + \sum_{i,k=1}^{j+1} \langle w_i,\widetilde{w}_k\rangle^2
$$
Now, once a generic curve $f:I\ra \mathds{Q}_n$ is given, we can
associate to each $s\in I$ the Minkovski $(j+2)$-space $V_j(s)$
giving rise to the osculating $j$-sphere \eqref{oscsphere}.
Setting $T_j(s,h)=T_j(V_j(s+h),V_j(s))$ and evaluating the Taylor
polynomial in $h$ around $0$, the authors in \cite{romerofuster}
define the conformal invariants $k_j$ as
\begin{equation}\label{defink}
k_j(s) = \sqrt{-\frac 12 \left.\frac{\partial^2T_j}{\partial h^2}\right|_{h=0}} \qquad 2\le j\le n-1,
\end{equation}
and they show (\cite{romerofuster}, p.381 Corollary 5) that
\begin{equation}\label{teorome}
k_j(s)\equiv \frac{r_{j-1}(s)r_{j+1}(s)\tau_{j+1}(s)}{r_j(s)^2\left[(k'(s))^2+k^2(s)\tau_2(s)^2\right]^{1/4}}.
\end{equation}
Here we prove
\begin{proposition}\label{rome}
For every $2\le j \le n-1$, the conformal invariant $\mu_j>0$ is
given by
$$
\mu_j(s)\equiv \frac{r_{j-1}(s)r_{j+1}(s)\tau_{j+1}(s)}{r_j(s)^2\left[(k'(s))^2+k^2(s)\tau_2(s)^2\right]^{1/4}},
$$
where $r_i$ is the radius of the osculating $i$-sphere.
\end{proposition}
\begin{proof}
The following set $\{w_0,w_k,w_{j+1}\} \subset \erre^{n+2}$:
$$
w_0= \frac{e_0+e_{n+1}}{\sqrt{2}}, \qquad w_k=e_k \quad \text{for }1\le k \le j, \qquad
w_{j+1}=\frac{e_0-e_{n+1}}{\sqrt{2}},
$$
gives a basis for the osculating $j$-sphere which diagonalizes the
Lorentzian inner product.
Differentiating the expression for $T_j(s,h)$, using $\langle
\dot{w}_a,w_b\rangle=-\langle \dot{w}_b,w_a\rangle$, $\langle
\ddot{w}_a,w_a\rangle=-|\dot{w}_a|^2$, $0\le a\le n$ we get
$T_j(s,0)=j+2$ and
$$
\begin{array}{lcl}
\dfrac{\partial T_j}{\partial h}(s,0) & = &  2\langle w_0,w_0\rangle\langle\dot{w}_0,w_0\rangle
- 2\sum_{i=1}^{j+1} \langle w_0,w_i\rangle\Big(\langle \dot{w}_0,w_i\rangle \\[0.3cm]
& & + \langle w_0, \dot{w}_i\rangle\Big) +
2\sum_{i,k=1}^{j+1} \langle w_i, w_k\rangle \langle w_i,\dot{w}_k\rangle = 0 \\[0.3cm]
\dfrac{\partial^2 T_j}{\partial h^2}(s,0) & = & -2\langle\ddot{w}_0,w_0\rangle
- 2\sum_{i=1}^{j+1} \langle \dot{w}_0,w_i\rangle^2 + 2 \sum_{i=1}^{j+1} \langle \ddot{w}_i,w_i\rangle \\[0.3cm]
& & - 2\sum_{i=1}^{j+1} \langle w_0, \dot{w}_i\rangle^2 +
\sum_{i,k=1}^{j+1} 2 \langle w_i,\dot{w}_k\rangle^2 \\[0.3cm]
& = & 2|\dot{w}_0|^2
- 4\sum_{i=1}^{j+1} \langle \dot{w}_0,w_i\rangle^2 - 2 \sum_{i=1}^{j+1} |\dot{w}_i|^2 + \sum_{i, k}
2 \langle w_i,\dot{w}_k\rangle^2 \\[0.3cm]
& = & 2\Big(|\dot{w}_0|^2
- \sum_{i=1}^{j+1} \langle \dot{w}_0,w_i\rangle^2\Big) \\[0.3cm]
& & - 2 \sum_{i=1}^{j+1}\Big(
|\dot{w}_i|^2 +\langle \dot{w}_i,w_0\rangle^2 - \sum_{k=1}^{j+1} \langle \dot{w}_i,w_k\rangle^2\Big).
\end{array}
$$
From the Frenet-Serret equations \eqref{frenetserretn} we
immediately deduce that, except for the term involving
$\dot{w}_j$, each addendum of the two sums in the RHS is zero
since $\dot{w}_0, \dot{w}_i\in <w_0,w_1,\ldots, w_{j+1}>$, $2\le i
\le j-1$; from $\dot{w}_{j} = -\mu_{j-1} w_{j-1} + \mu_{j}e_{j+1}$
we conclude
$$
\dfrac{\partial^2 T_j}{\partial h^2}(s,0) = -2 \Big(
|\dot{w}_j|^2 +\langle \dot{w}_j,w_0\rangle^2 - \sum_{k=1}^{j+1} \langle \dot{w}_j,w_k\rangle^2\Big) = -2\mu_j^2
$$
Hence, by \eqref{defink}, $\mu_j\equiv k_j$ for every $2\le j\le
n-1$; putting this together with \eqref{teorome}, we conclude.
\end{proof}
With the aid of the previous constructions, we can supply a simple
proof of an interesting property already shown in \cite{CSWebb}
with different methods. Let $f$ be a closed $1$-generic curve in
$\mathds{Q}_3$. Then, we define the \textbf{total twist} of $f$ as
the normalized integral
$$
\mathrm{Tw}(f) = \frac{1}{2\pi} \int_I \tau \,\di s_e \qquad (\text{mod } \mathds{Z}).
$$
\begin{proposition}\label{tottwist}
Let $f: I \ra \mathds{Q}_3$ be a closed, $1$-generic curve. Then,
\begin{equation}\label{twist}
\frac{1}{2\pi}\int_I\mu_2 \,\di s \equiv \frac{1}{2\pi}\int_I \tau\,
\di s_e \qquad (\mathrm{mod} \; \mathds{Z}).
\end{equation}
Therefore, the total twist $\mathrm{Tw}(f)$ is a conformal
invariant.
\end{proposition}
\begin{proof}
We perform most of the proof for $f:I \ra\mathds{Q}_n$, $n\ge3$.
We assume $f$ to be generic. Then $Y=\tau_2e_3$ with respect to
the Euclidean frame $e$. Therefore, we can write the vector
$Z/\|Z\|$ in \eqref{defZ} as $\cos\theta e_2 + \sin\theta e_3$ for
some $\theta$ globally defined on $I$. Using \eqref{coseB2},
$^t(B^\alpha_2) = {}^t(\cos\theta,\sin\theta, 0)$, and a suitable
matrix $B$ satisfying \eqref{coseB2} is given by
$$
B = \left(\begin{array}{ccc} \cos\theta & -\sin\theta & 0 \\
\sin\theta & \cos\theta & 0 \\
0 & 0 & I_{n-3}\end{array}\right).
$$
A computation shows that
$$
\begin{array}{lcl}
\dfrac{\nabla^e(Z/\|Z\|)}{\di s_e} & = & \displaystyle -\theta'\sin\theta e_2 + \cos\theta \frac{\nabla^e e_2}{\di s_e}
+\theta'\cos\theta e_3 + \sin\theta \frac{\nabla^e e_3}{\di s_e}\\[0.3cm]
& = & \displaystyle -\sin\theta(\theta'+\tau_2)e_2 + \cos\theta(\theta'+\tau_2)e_3 + \tau_3\sin\theta e_4.
\end{array}
$$
and in \eqref{phib2}
\begin{equation*}
\hat{\phi}^b_2  =  \sq{-B_b^2 \sin\theta(\theta'+\tau_2) +
B_b^3 \cos\theta(\theta'+\tau_2)+ B_b^4\tau_3\sin\theta}\di s_e.
\end{equation*}
Therefore
\begin{equation*}
 \hat{\phi}^3_2 = (\theta'+\tau_2)\di s_e,\qquad \hat{\phi}^4_2 = \tau_3\sin\theta\di s_e, \qquad \hat{\phi}^c_2=0 \quad 5\le c\le n.
\end{equation*}
If $n=3$, then $\hat{\phi}^3_2=\mu_2\di s$ and the computation
above holds with the only requirement of $1$-genericity. Since $f$
is closed, the conclusion follows by integrating $\hat{\phi}^3_2$
over $I$.
\end{proof}
\begin{remark}
\emph{To the best of our knowledge, a first proof of the conformal
invariance of the total twist appeared in \cite{banwhite}. A
straightforward application of Proposition \ref{tottwist} is the
following: if $f$ is included in some $\mathds{Q}_2$, then
$(2\pi)^{-1}\int_I \tau \di s_e$ is an integer. Indeed, in this
case $\mu_2\equiv 0$. As a matter of fact, much more is true: a
surface $\Sigma\subset \erre^3$ is an Euclidean $2$-sphere if and
only if $\int_I\tau \di s_e = 0$ for every closed curve $f: I \ra
\Sigma$ (Scherrer theorem, \cite{scherrer}, subsequently
generalized in \cite{saban}). The ``only if'' part is known as
Fenchel-Jacobi theorem \cite{fencheljacobi}, and a different proof
of it appears also in (\cite{rogen1}, Corollary 21).
Fenchel-Jacobi theorem has been generalized in \cite{white1} for
embedded manifolds $M^n\ra \esse^{2n} \subset \erre^{2n+1}$ when
$n$ is odd (Theorem 10).}
\end{remark}
%
%
%
%
%
%
%
%
\section{The Euler-Lagrange equations for the conformal arclength}\label{sezionearco}
Let $f:I\subset \erre \ra \mathds{Q}_n$ be an immersion. In the
previous section we defined a natural parameter $s$, called
conformal arclength, for second order frames along $f$. If
$\mathcal{D} \subset \subset I$ is a compact subinterval, the
\textbf{conformal arclength functional} on $\mathcal{D}$ is the
integral
\[
G_\mathcal{D}(f) = \int_\mathcal{D} \di s = \int_\mathcal{D} \sqrt[4]{\sum_\alpha \pa{p^\alpha}^2}\phi^1_0.
\]
We shall study the extremal points of the conformal arclength functional, which we shall call \textbf{conformal geodesics}.\\
To this end we first need to introduce the concept of admissible
variation, that is a smooth map $v : I \times (-\eps, \eps) \ra
\mathds{Q}_n$ such that
\begin{enumerate}
    \item[-] $f(\cdot, t) : I \ra \mathds{Q}_n$ is an
        immersion for every $t \in (-\eps, \eps)$;
    \item[-] $v(\cdot,0)=f$;
    \item[-] $v(x,t) = f(x)$ for $x$ outside a compact set $K
        \subset I$ and $t\in (-\eps,\eps)$.
\end{enumerate}
Let $i_t$ denote the standard inclusion $s\in I\ra (s,t)\in I
\times (\eps,\eps)$, and for simplicity write $f_t$ instead of $v(\cdot,t)$.\\
Since $\di s$ is only $C^{0,1/2}$ around $1$-degenerate points, to
avoid problems we assume $f$ to be globally $1$-generic. Since
$n\ge 3$, this assumption is not restrictive.
Up to choosing $\eps$ sufficiently small, we can assume that every curve $f_t$ is $1$-generic. \\
We also need to consider frames with good properties along every
immersion $f_t$. Therefore we define a \textbf{special third order
frame along an admissible variation} $v$ as a frame $e: U \times
(-\eps,\eps) \ra \mathds{Q}_n$, where $U\subset I$ is an open set,
such that:
\begin{itemize}
\item[-] $e$ is a frame along $v$, that is, $\pi\circ e = v$;
\item[-] $e(x,t)= e(x,0)$ for every $x\in U\setminus K$, $t\in (-\eps,\eps)$;
\item[-] $e_t= e(\cdot,t)$ is a special third order frame along $f_t$.
\end{itemize}
For every $(p,t)\in U\times (-\eps,\eps)$, the Maurer-Cartan form $\phi= e^*\Phi$ can be decomposed as
\begin{equation}\label{spezzphip}
\phi_{(p,t)} = \phi(t) + \Lambda(p,t)\di t,
\end{equation}
where $\phi(t)$ is just $e_t^*\Phi$ and satisfies $\phi(t)\left(\frac{\partial}{\partial t}\right) = 0$, while $\Lambda(p,t)$ is a smooth $\mab(n)$-valued map. With respect to special third order frames along $v$, $\phi$ satisfies
\begin{align}
&\phi^0_0 = \Lambda_0^0 \di t;\qquad \phi^\alpha_0 = \Lambda_0^\alpha \di t;\label{3.3}\\
\nonumber&\phi^\alpha_1 = \Lambda_1^\alpha \di t;\qquad\phi^0_b = \Lambda^0_b \di t;\\
\nonumber&\phi^0_1 = \mu_1\phi^1_0(t)+\Lambda^0_1 \di t=\mu_1\phi^1_0 + (\Lambda^0_1 - \mu_1\Lambda^1_0)\di t;\\
\nonumber&\phi^0_2 = \phi^1_0(t) + \Lambda^0_2 \di t = \phi^1_0 + (\Lambda^0_2 - \Lambda^1_0)\di t;\\
&\phi^b_2=q^b\phi^1_0(t)+\Lambda^b_2\di t=q^b\phi^1_0+(\Lambda^b_2-q^b\Lambda^1_0)\di t,\nonumber
\end{align}
where we have fixed the index convention $a,b,c,d \in \{3,\ldots, n\}$. For convenience, we set
\begin{align}
&\lambda^0_0=\Lambda^0_0, \qquad\lambda^\alpha_0= \Lambda^\alpha_0, \label{notaz}\\
&\lambda^\alpha_1= \Lambda^\alpha_1,\qquad \lambda^0_2=
\Lambda^0_2 - \Lambda^1_0, \nonumber\\
&\lambda^0_b=\Lambda^0_b, \qquad \lambda^b_2=\Lambda^b_2-q^b\Lambda^1_0\nonumber.
\end{align}
Although, in literature, the possibility of constructing
variations $v$ and special type of frames $e$ with arbitrary
initial data $\{\lambda^\alpha_0(p)\}$ is widely used without
proof and is a standard fact, we have found no accessible and
complete proof. The question is not straightforward since,
sometimes, locally defined frames are used to compute variations
of functionals whose support may be, a priori, not contained in
that of the frame; of course, one can take collections
$\{\lambda^\alpha_0\}$ with arbitrarily small support when
computing the Euler-Lagrange equations, but since the frame
depends on the variation and this latter on the set
$\{\lambda^\alpha_0\}$, the question whether we can assume to have
the support of the variation contained in that of the frame raises
a doubt. This is important when integrating, because we need to be
sure that the exact terms vanish by Stokes' theorem. If the
support of the variation is not contained in that of the frame,
one should show that the exact terms appearing are the
differential of \emph{globally defined objects}. Checking the
invariance of such expressions is often a lengthy and very
complicate computation, although straightforward, so a different
strategy must be used. A complete proof of the proposition below
is not difficult but requires some care, and we postpone it to the
Appendix.
\begin{proposition}\label{existvariation}
For every $p\in I$ there exists an open neighbourhood $U$ of $p$
such that the following holds: for every collection of $(n-1)$
smooth functions $\lambda^\alpha \in C^\infty(I)$ with compact
support $C$ included in $U$, there exist $\eps$ sufficiently
small, a variation $v: I \times (-\eps,\eps)\ra \mathds{Q}_n$ and
a special third order frame $e: U \times (-\eps,\eps)\ra \mob(n)$
along $v$ such that $\lambda^\alpha_0(p,0)= \lambda^\alpha(p)$ for
every $p\in U$.
\end{proposition}
\begin{remark}
\emph{Observe that the key fact in the above proposition is that
we can take the same neighbourhood $U$ for every collection
$\lambda^\alpha$, although $\eps$ depends on the chosen
collection.}
\end{remark}
Define $s$ to be the conformal arclength parameter of the
immersion $f_0=f$, that is, $\di s =\phi^1_0(0)$, and set
$X_t=q^b(t)e_b(t)$. Write $X$ instead of $X_0$ for notational
convenience. We are ready to prove the following result:
\begin{theorem}\label{eullagrangecaso1}
When $n\ge3$, a $1$-generic immersion $f:I\ra \mathds{Q}_n$ is a
conformal geodesic if and only if the following Euler-Lagrange
equations are satisfied:
\begin{equation}\label{EuleroLagrangecurve}
\left\{ \begin{array}{ll}
\dfrac{\di\mu_1}{\di s} + \dfrac{3}{2} \dfrac{\di\abs{X}^2}{\di s} =0; \\[0.5cm]
\dfrac{\nabla^2X}{\di s^2} - X(\abs{X}^2+2\mu_1) = 0,
\end{array}\right.
\end{equation}
\end{theorem}
\begin{proof} By the fundamental theorem of the calculus of variations, it is enough to consider variations $v$ with arbitrarily small support. Therefore, the above existence theorem applies and we can work with a global frame containing all the support of $v$. Moreover, since the variation is compactly supported in $\mathcal{D}$, in the decomposition
$$
\phi = \phi(t) + \Lambda \di t,
$$
the components of $\Lambda$ are compactly supported in $\mathcal{D}$.\\
Differentiating some of the equations in \eqref{3.3} and using \eqref{notaz}, the structure equations and Cartan's lemma, we get
\begin{eqnarray}
 \lambda_1^\alpha\phi^1_0 -\di\lambda_0^\alpha - \lambda_0^\beta\phi^\alpha_\beta + \lambda_0^\alpha\phi^0_0
  & = & f^\alpha \di t; \label{dilambdaalpha0}\\[0.2cm]
\lambda^0_\alpha\phi^1_0 - \di\lambda_1^\alpha - \lambda_1^\beta\phi^\alpha_\beta + \lambda_0^\alpha\phi^0_1
 & = & g^\alpha \di t; \label{dilambdaalpha1} \\[0.2cm]
 \di \lambda^0_2 +\lambda^0_b\phi^2_b -
\lambda_1^2\phi^0_1 + \lambda^0_2\phi^0_0& = & 2\lambda^0_0 \phi^1_0 + \eta^2 \di t, \label{dilambda02}\\[0.2cm]
\di \lambda^0_b +\lambda^0_2\phi^b_2 +\lambda^0_c\phi^b_c-
\lambda_1^b\phi^0_1 + \lambda^0_b\phi^0_0& = & \lambda^b_2 \phi^1_0 + \eta^b \di t, \label{dilambda0b}
\end{eqnarray}
for some smooth coefficients $f^\alpha$, $g^\alpha$, $\eta^\alpha$, compactly supported in $\mathcal{D}$. \\
Since $\mathcal{D}$ is relatively compact and
$\Omega$ is smooth on $I$, we get
\begin{equation}\label{primop}
\frac{d}{dt}\left.\!\!\!\frac{}{}\right|_{t=0}G_\mathcal{D}(f_t) = \int_\mathcal{D}
\pa{\mathcal{L}_{\frac{\partial}{\partial t}}\phi^1_0(t)}\left.\!\!\!\frac{}{}\right|_{t=0}= \int_\mathcal{D}
\pa{i_{\frac{\partial}{\partial t}} \di \phi^1_0(t) +\di \pa{i_{\frac{\partial}{\partial t}} \phi^1_0(t)}}\left.\!\!
\!\frac{}{}\right|_{t=0}
\end{equation}
where $|_{t=0}$ stands for $i_0^*$. Since
$\phi^1_0(t)\pa{\frac{\partial}{\partial t}}=0$,
\begin{equation}
\frac{d}{dt}\left.\!\!\!\frac{}{}\right|_{t=0}G_\mathcal{D}(f_t) = \int_\mathcal{D} \pa{i_{\frac{\partial}{\partial t}}
\di \phi^1_0(t)}\left.\!\!\!\frac{}{}\right|_{t=0}.
\end{equation}
Using $\phi^1_0(t)=\phi^1_0-\Lambda^1_0 \di t$ we compute
\begin{equation}
i_{\frac{\partial}{\partial t}} \pa{\di \phi^1_0(t)} =  i_{\frac{\partial}{\partial t}}\pa{\phi^0_0\wedge \phi^1_0-\di \Lambda^1_0 \wedge \di t}= i_{\frac{\partial}{\partial t}}\pa{\phi^0_0\wedge \phi^1_0}-\frac{\partial\Lambda^1_0}{\partial t}\, \di t + \di\Lambda^1_0
\end{equation}
and observe that the term with $\di t$ will vanish when restricted to $t=0$. Moreover, the differential $\di\Lambda^1_0\left.\!\!\!\frac{}{}\right|_{t=0}$ has compact support in $\mathcal{D}$, so its integral vanishes by Stokes' theorem. Therefore
\begin{equation*}
\frac{d}{dt}\left.\!\!\!\frac{}{}\right|_{t=0}G_\mathcal{D}(f_t) =
\int_\mathcal{D}\sq{i_{\frac{\partial}{\partial t}}\pa{\phi^0_0\wedge \phi^1_0}}\left.\!\!\!\frac{}{}\right|_{t=0}
=\int_\mathcal{D}\sq{\lambda^0_0\phi^1_0}_{t=0}.
\end{equation*}
Using \eqref{dilambda02}, the integrand becomes, at $t=0$,
\begin{equation*}
\lambda^0_0\phi^1_0=\frac12\pa{\di \lambda^0_2 +\lambda^0_b\phi^2_b -
\lambda_1^2\phi^0_1 + \lambda^0_2\phi^0_0}.
\end{equation*}
For ease of notation, we write $\omega \equiv \eta$ to mean that the form $\omega$ differs from $\eta$ by the differential of some compactly supported function, and we omit specifying the restriction to $t=0$ when it is clear from the context. \par
The process is now a simple integration by parts, where we get rid of the compactly supported exact forms as they appear:
$$
2\lambda^0_0\phi^1_0 \equiv \lambda^0_b\phi^2_b -
\lambda_1^2\phi^0_1 = -\lambda^0_b q^b\phi^1_0 - \lambda^2_1\phi^0_1 \qquad \text{at } t=0.
$$
Substituting \eqref{dilambdaalpha1}, \eqref{dilambdaalpha0} and integrating by parts, we obtain at $t=0$
\begin{align*}
2\lambda^0_0\phi^1_0 \equiv & -q^b\left( \di \lambda^b_1 +\lambda^2_1 \phi^b_2 + \lambda^c_1\phi^b_c - \lambda^b_0\phi^0_1\right)- \lambda^2_1\mu_1 \phi^1_0 \\
\equiv &\ \lambda^b_1 \di q^b - q^bq^b\lambda^2_1\phi^1_0 - q^b \lambda^c_1 \phi^b_c + q^b\lambda^b_0 \mu_1\phi^1_0- \mu_1\left(\di \lambda^2_0 + \lambda^b_0 \phi^2_b\right)\\
\equiv &\ \lambda^b_1 \di q^b - q^bq^b\left(\di \lambda^2_0 + \lambda^c_0 \phi^2_c\right) - q^b \lambda^c_1 \phi^b_c + q^b\lambda^b_0 \mu_1\phi^1_0\\
&+ \lambda^2_0 \di \mu_1 + \lambda^b_0 \mu_1q^b\phi^1_0\\
 \equiv & \ \lambda^b_1 \di q^b +\lambda^2_0 \di(q^bq^b) + (q^bq^b)\lambda^c_0 q^c\phi^1_0 -q^c \lambda^b_1 \phi^c_b \\
 &+ 2q^b\lambda^b_0 \mu_1\phi^1_0 + \lambda^2_0 \di \mu_1\\
 = &\ \lambda^b_1\left( \di q^b + q^c \phi^b_c\right) + \lambda^2_0 \di (q^bq^b) + (q^bq^b)q^c\lambda^c_0 \phi^1_0\\
 &+ 2q^b\lambda^b_0 \mu_1\phi^1_0 + \lambda^2_0 \di \mu_1.
\end{align*}
Since $\phi^1_0 = \di s$ at $t=0$, we can write $\di \mu_1 = \frac{\di\mu_1}{\di s} \phi^1_0$.
Moreover, recalling the definition of $X$, we have
$$
q^bq^b = \abs{X}^2, \qquad \di q^b + q^c \phi^b_c = \left(\nabla X\right)^b \phi^1_0.
$$
Thus we have
\begin{align*}
2\lambda^0_0\phi^1_0 \equiv &   \left(\nabla X\right)^b \left(\di \lambda^b_0 + \lambda^2_0 \phi^b_2 + \lambda^c_0 \phi^b_c\right)+ \lambda^2_0 \di \abs{X}^2  \\
& + \abs{X}^2 q^c\lambda^c_0 \phi^1_0+ 2q^b\lambda^b_0 \mu_1\phi^1_0 + \lambda^2_0 \dfrac{\di \mu_1}{\di s} \phi^1_0\\
 \equiv & -\lambda^b_0 \left(\di(\nabla X)^b + (\nabla X)^c \phi^b_c\right) + (\nabla X)^b q^b\lambda^2_0 \phi^1_0 + \lambda^2_0 \di \abs{X}^2 \\
& + \abs{X}^2 q^c\lambda^c_0 \phi^1_0+ 2q^b\lambda^b_0 \mu_1\phi^1_0 + \lambda^2_0 \dfrac{\di \mu_1}{\di s} \phi^1_0.
\end{align*}
Noting that
$$
\begin{array}{l}
\di \abs{X}^2 = \dfrac{\di \abs{X}^2}{\di s} \phi^1_0, \qquad
(\nabla X)^b q^b = \langle \nabla X,X\rangle = \frac{1}{2} \di \abs{X}^2, \\[0.3cm]
\di(\nabla X)^b + (\nabla X)^c \phi^b_c = (\nabla^2X)^b \phi^1_0,
\end{array}
$$
the RHS becomes
$$
\Big[\lambda^2_0 \Big( \frac{3}{2} \dfrac{\di \abs{X}^2}{\di s} + \dfrac{\di \mu_1}{\di s}\Big) +
\lambda^b_0\Big( - (\nabla^2X)^b + \abs{X}^2 q^b + 2q^b\mu_1\Big)\Big] \phi^1_0.
$$
By the arbitrariness of $\lambda^\alpha_0(p,0)$, and since
$$
\frac{\nabla^2X}{\di s^2} = (\nabla^2X)^b e_b,
$$
the Euler-Lagrange equations of the conformal geodesics are
\eqref{EuleroLagrangecurve}, as required.
\end{proof}
We now go deeper in investigating the solutions of equations
\eqref{EuleroLagrangecurve}. First, we observe that $\di e =
e\phi$ for a special third order frame applied to the vector field
$\frac{\di}{\di s}$ read
\begin{equation}\label{frenetdai}
\begin{array}{l}
\dot{e}_0  =  e_1; \qquad \dot{e}_1  =  \mu_1 e_0 + e_{n+1}; \qquad
\dot{e}_{n+1}  =  \mu_1 e_1 + e_2; \\[0.1cm]
\dot{e}_2  =  e_0 + X; \qquad
\dot{e}_b  =  - q^b e_2 + \phi^c_b\left(\frac{\di}{\di s}\right) e_c,
\end{array}
\end{equation}
where the dot denotes the derivative of the components with
respect to the parameter $s$. In other words, we consider the
vector bundle $W = <e_0,e_A,e_{n+1}>$ associate to the principal
bundle $\mob(n)\ra \mathds{Q}_n$, endowed with the Lorentzian
metric and we see $\Theta$ as a vector subbundle of $W$, with the
induced (Riemannian) metric and a compatible connection $\nabla$.
From the above equations we get
\begin{equation}\label{Xpunto e nablaX}
\dot{X} = \dot{q}^be_b + q^b \dot{e}_b= \left(\dot{q}^b + q^c \phi^b_c\left(\tfrac{\di}{\di s}\right)\right)e_b -
q^bq^b e_2 =
\frac{\nabla X}{\di s} - \abs{X}^2 e_2.
\end{equation}
Differentiating once more, we get
\begin{equation}\label{derivatesecondeX}
\dot{\left(\frac{\nabla X}{\di s}\right)} = \frac{\nabla^2X}{\di s^2} - \frac 12 \frac{\di}{\di s}\abs{X}^2 e_2.
\end{equation}
Define
\begin{equation}
V(s) = < e_0(s), e_1(s), e_2(s), e_{n+1}(s), X(s), \frac{\nabla X}{\di s}(s) >,
\end{equation}
and observe that $4 \le \dim V(s) \le 6$, and that $V(s)$ is a
Lorentzian subspace. We can prove the following result:
\begin{theorem}\label{dimensriduz}
Let $f:I\ra \mathds{Q}_n$ be a $1$-generic conformal geodesic.
Then, $V(s)$ is a time-like vector space independent of $s$, which
we call $V$. Moreover, $V$ identifies a conformal sphere of
dimension $\dim V-2$ containing the whole immersion $f$.
\end{theorem}
\begin{proof}
Choose a special third order frame along $f$ and write $E(s)$ for
the $(n+2)\times 6$ matrix $\pa{e_0|e_1|e_2|X|\frac{\nabla X}{\di
s}| e_{n+1}}$. Note that $\mathrm{Rank}(E)\ge 4$. Since $f$ is a
conformal geodesic, from \eqref{derivatesecondeX} we obtain
\begin{equation}\label{derivatesecondeX2}
\dot{\left(\frac{\nabla X}{\di s}\right)} = X(\abs{X}^2+2\mu_1) - \frac 12 \frac{\di}{\di s}\abs{X}^2 e_2.
\end{equation}
Integrating the first Euler-Lagrange equation we get
\begin{equation}\label{integrprima}
\abs{X}^2 = -\frac 23 \mu_1 + C_1,
\end{equation}
for some constant $C_1\in \erre$. Using \eqref{frenetdai},
\eqref{Xpunto e nablaX}, and \eqref{integrprima} we can see that
$E(s)$ satisfies
\begin{equation}\label{sistlineare}
\dot{E}(s) = E(s)A(s), \qquad \text{where}
\end{equation}
\begin{equation}\label{deaA}
A(s) = \left(\begin{array}{cccccc}
0 & \mu_1 & 1 & 0 & 0 & 0 \\[0.1cm]
1 & 0 & 0 & 0 & 0 & \mu_1 \\[0.1cm]
0 & 0 & 0 & \frac 23 \mu_1 - C_1 & \frac 13 \frac{\di\mu_1}{\di s} & 1 \\[0.1cm]
0 & 0 & 1 & 0 & \frac 43 \mu_1 + C_1 & 0 \\[0.1cm]
0 & 0 & 0 & 1 & 0 & 0 \\[0.1cm]
0 & 1 & 0 & 0 & 0 & 0
\end{array}\right).
\end{equation}
Now, \eqref{sistlineare} is a linear system with matrix $A$
independent of $E$. By the existence-uniqueness theorem for linear
ODEs, the linear independence of the span of the columns of $E$ is
preserved, hence the rank of $E$ is constant along $I$. Moreover,
equation $\dot{E}(s)=E(s)A(s)$ implies that $V(s)$ is a vector
space $V$ independent of $s$. It follows that the intersection of
$V$ with the positive light cone projects to a conformal sphere of
dimension $2\le \dim V -2 \le 4$ containing $[e_0]$. This
concludes the proof.
\end{proof}
\begin{theorem}\label{teoreduz}
Every conformal geodesic $f: I \ra \mathds{Q}_n$ is included in
some conformal $4$-sphere $\mathds{Q}_4\subset \mathds{Q}_n$.
\end{theorem}
The following theorem examines each of the three possible values
taken by $\dim V$.
\begin{theorem}\label{reductioncod}
Let $f:I \ra \mathds{Q}_n$ be a $1$-generic conformal geodesic.
\begin{itemize}
\item[$(i)$] if $\dim V=4$, then $f$ is a totally $2$-degenerate curve of constant curvature $\mu_1$;
\item[$(ii)$] if $\dim V=5$, then $f$ is a $2$-generic and totally
$3$-degenerate curve. Choosing a fourth order frame, the equations
satisfied by the two curvatures $\mu_1$ and $\mu_2$ are
\begin{equation}\label{eultre}
\left\{ \begin{array}{ll}
\dot{\mu}_1+ 3 \mu_2\dot{\mu}_2=0 \\[0.2cm]
\ddot{\mu}_2 = \mu_2^3 + 2\mu_1\mu_2,
\end{array}\right.
\end{equation}
where the dot denotes the derivative with respect to the arclength parameter;
\item[$(iii)$] if $\dim V =6$, then $f$ is a $3$-generic and
totally $4$-degenerate curve. For a fifth order frame, the
curvatures $\mu_1,\mu_2,\mu_3$ satisfy the system
\begin{equation}\label{eulquattro}
\left\{ \begin{array}{ll}
\dot{\mu}_1+ 3 \mu_2\dot{\mu}_2=0 \\[0.2cm]
\ddot{\mu}_2 = \mu_2^3 + 2\mu_1\mu_2 + \mu_2\mu_3^2\\[0.2cm]
2\dot{\mu}_2\mu_3 + \mu_2 \dot{\mu}_3 =0.
\end{array}\right.
\end{equation}
\end{itemize}
Conversely, condition $\mu_1=\mathrm{constant}$ characterizes
totally $2$-degenerate conformal geodesic, a solution
$\{\mu_1,\mu_2\}$ of \eqref{eultre} with $\mu_2\neq 0$ for every
$s\in I$ characterizes $2$-generic and totally $3$-degenerate
curves in $\mathds{Q}_3$ and a solution of \eqref{eulquattro} with
$\mu_2,\mu_3\neq 0$ for every $s\in I$ characterizes $3$-generic
conformal geodesic on $\mathds{Q}_4$.
\end{theorem}
\begin{proof}
Observe that if $X(p)\neq 0$ for some $p$ then, by its very definition, $X$ is linearly independent of $e_0,e_1,e_2,e_{n+1}$. The same holds for $\nabla X/\di s$. By Proposition \ref{casototalmdegenere} and Theorem \ref{dimensriduz}, $\dim V = 4$ if and only if both $X$ and $\nabla X/\di s$ vanish identically. In this case, the first Euler-Lagrange equation becomes $\mu_1$ constant.\\
In case $\dim V=5$, then the curve is $2$-generic on the whole $I$ and totally $3$-degenerate, so
that $X(t)\neq 0$ for every $t\in I$. Taking a fourth order frame,
$$
X=\mu_2 e_3 \quad \text{with }  \mu_2> 0 \text{ on } I; \qquad
\dfrac{\nabla e_3}{\di s} = \phi^b_3\Big(\dfrac{\di}{\di s}\Big) e_b =0,
$$
where the last equality follows from the totally $3$-degeneracy
($\phi^b_3=0$). Differentiating, we obtain $\nabla^2 X/\di
s^2 = \ddot{\mu}_2 e_3$, and the pair of Euler-Lagrange equations \eqref{eultre} are readily obtained.\\
In case $\dim V=6$, the curve is $3$-generic but totally
$4$-degenerate ($\phi^c_4=0$ for every $c\ge 5$), so that in a
fifth order frame
$$
\begin{array}{l}
X=\mu_2 e_3 \quad \text{with }  \mu_2> 0 \text{ on } I;\\[0.2cm]
\dfrac{\nabla e_3}{\di s} = \phi^b_3\Big(\dfrac{\di}{\di s}\Big) e_b =\mu_3 e_4
\quad \text{with }  \mu_3> 0 \text{ on } I;\\[0.4cm]
\dfrac{\nabla e_4}{\di s} = \phi^3_4\Big(\dfrac{\di}{\di s}\Big) e_3 +
 \phi^c_4\Big(\dfrac{\di}{\di s}\Big) e_c =-\mu_3 e_3.
\end{array}
$$
Differentiating, we get
$$
\dfrac{\nabla^2X}{\di s^2} = \left[\ddot{\mu}_2-\mu_2\mu_3^2\right]e_3 +
\left[2\dot{\mu}_2\mu_3 + \mu_2 \dot{\mu}_3\right]e_4,
$$
from which formulas \eqref{eulquattro} follow at once. The converse is immediate and follows from the Cartan-Darboux rigidity theorem \ref{rigidita}.
\end{proof}
\section{Integration of the equations of
motion}\label{sez_equazionimoto}
First of all, we observe that in case the curve is totally
$3$-degenerate, \eqref{eultre} coincides with the system in
\cite{musso}. We therefore limit ourselves to considering the
$3$-generic case
\begin{equation}\label{eulquattro2}
\left\{ \begin{array}{ll}
\dot{\mu}_1+ 3 \mu_2\dot{\mu}_2=0 \\[0.2cm]
\ddot{\mu}_2 = \mu_2^3 + 2\mu_1\mu_2 + \mu_2\mu_3^2\\[0.2cm]
2\dot{\mu}_2\mu_3 + \mu_2 \dot{\mu}_3 =0
\end{array}\right.
\end{equation}
on a subset $I\subset \erre$. Integrating the first and third
equation, remembering that $\mu_2,\mu_3>0$ and substituting into
the second one we get
\begin{equation}\label{relazalgebrichecurv}
\left\{ \begin{array}{lll}
\mu_1 = -\dfrac 32 \mu_2^2 + C_1, & \quad C_1 \in \erre \\[0.4cm]
\mu_2^2\mu_3 = C_2, & \quad C_2\in \erre, \ C_2 \neq 0 \\[0.2cm]
\ddot{\mu}_2 = -2\mu_2^3 + 2C_1 \mu_2 + \dfrac{C_2^2}{\mu_2^3},
\end{array}\right.
\end{equation}
therefore the constancy of any of the curvatures implies that the
geodesic $f$ has all the curvatures constant. Since every solution
of \eqref{eulquattro2} is real analytic, $\mu_2$ has either
isolated stationary points or it is constant. \emph{From now on,
we assume that each curvature is not constant}. In this case,
multiplying the second equation by $\dot{\mu}_2$ and integrating
we obtain an equivalent differential equation, expressing a
conservation of energy:
\begin{equation}\label{energy}
\frac 12 \dot{\mu}_2^2 + \frac{1}{2}\mu_2^4 - C_1 \mu_2^2 + \frac{C_2^2}{2\mu_2^2} = \frac{C_3}{2},
\end{equation}
for some $C_3\in \erre$. We are eventually led to solve
\eqref{energy} for a positive function $\mu_2$ and an admissible
triple of real constants $C_1,C_2,C_3$. Multiplying by $\mu_2^2$,
taking square roots and changing variables we get
\begin{equation}\label{elli}
\int^{\mu_2^2(s)}_{\mu_2^2(s_0)} \frac{\di t}{\sqrt{-t^3+2C_1 t^2 +C_3 t - C_2^2}} = s-s_0.
\end{equation}
This can be solved by using elliptic functions. Denote with
$\xi_-, \xi_1,\xi_2$ the complex roots of the polynomial
\begin{equation}
P(t)= -t^3 +2C_1 t^2 + C_3 t - C_2^2,
\end{equation}
and note that $\xi_-\xi_1\xi_2=-C^2_2<0$, thus it can only happen
one of the following cases:
\begin{equation}\label{casiradici}
\begin{array}{ll}
(i) & \xi_-\in \erre, \ \xi_-<0, \ \xi_1 = \bar{\xi}_2\in
\mathds{C}\backslash \erre; \\[0.1cm]
(ii) & \xi_-,\xi_1,\xi_2\in \erre, \ \xi_- <\xi_1 \le \xi_2 < 0; \\[0.1cm]
(iii) & \xi_-,\xi_1,\xi_2\in \erre, \ \xi_- < 0 < \xi_1 = \xi_2; \\[0.1cm]
(iv) & \xi_-,\xi_1,\xi_2\in \erre, \ \xi_- < 0 < \xi_1 < \xi_2.
\end{array}
\end{equation}
Since the integral is between positive extremes and $\mu_2$ is not
constant, only case $(iv)$ is possible, therefore $\mu^2_2\in
[\xi_1,\xi_2]$ is a bounded function. Hereafter, we restrict to
the triples $(C_1,C_2,C_3)$ such that $P(t)$ has one negative and
two distinct positive solutions. Up to a translation of the
arclength parameter, and since the integral is finite around
$\xi_1$, we can assume that $\mu^2_2= \xi_1$ when $s=0$, so that
\eqref{elli} becomes
\begin{equation}\label{elli2}
\int^{\mu_2^2}_{\xi_1} \frac{\di t}{\sqrt{-(t-\xi_-)(t-\xi_1)(t-\xi_2)}} = s.
\end{equation}
The change of variables
$$
t = \xi_1 + \theta^2(\xi_2-\xi_1)
$$
Takes to the elliptic incomplete integral of first kind
\begin{equation}\label{elli2}
\frac{2}{\sqrt{\xi_1-\xi_-}}\int^{\sqrt{\frac{\mu_2^2-\xi_1}{\xi_2-\xi_1}}}_0
\frac{\di \theta}{\sqrt{1-\theta^2}\sqrt{1+\frac{\xi_2-\xi_1}{\xi_1-\xi_-}\theta^2}} = s.
\end{equation}
By (\cite{lawden} p. 51) we can apply formula
$$
\int^x_0 \frac{\di t}{\sqrt{b^2-t^2}\sqrt{a^2+t^2}} = \frac{1}{\sqrt{a^2+b^2}} \ \mathrm{sd}^{-1}
\left[\frac{x\sqrt{a^2+b^2}}{ab}, \frac{b}{\sqrt{a^2+b^2}}\right],
$$
where $0 \le x\le b, \ a>0$, with the suitable choices to obtain
\begin{equation}
\mu_2 = \sqrt{ \xi_1 +\frac{(\xi_2-\xi_1)(\xi_1-\xi_-)}{\xi_2-\xi_-}\left(\mathrm{sd}\left[s\sqrt{\frac{\xi_2-\xi_-}{\xi_1-\xi_-}},
\sqrt{\frac{\xi_2-\xi_1}{\xi_2-\xi_-}}\right]\right)^2}
\end{equation}
We summarize the result in the following
\begin{proposition}\label{soluzionecurvatura}
There exists a non-constant solution $(\mu_1,\mu_2,\mu_3)$,
$\mu_2>0$, $\mu_3>0$, of the system
\begin{equation}\label{sistenerg}
\left\{\begin{array}{l}
\displaystyle \mu_1 = - \frac 32 \mu_2^2+C_1 \\[0.3cm]
\mu_3 = C_2/\mu_2^2, \quad C_2 \neq 0 \\[0.1cm]
\displaystyle \frac 12 \dot{\mu}_2^2 + \frac{1}{2}\mu_2^4 - C_1 \mu_2^2 + \frac{C_2^2}{2\mu_2^2} = \frac{C_3}{2},
\end{array}\right.
\end{equation}
if and only if $(C_1,C_2,C_3)$ is an admissible triple, that is,
\begin{equation}\label{Pt}
P(t)= -t^3 +2C_1 t^2 + C_3 t - C_2^2
\end{equation}
has real roots $\xi_-<0<\xi_1<\xi_2$. In such case, $\mu_2$ is
given by the formula
\begin{equation}\label{intell}
\mu_2 = \sqrt{ \xi_1 +\frac{(\xi_2-\xi_1)(\xi_1-\xi_-)}{\xi_2-\xi_-}\left(\mathrm{sd}\left[s\sqrt{\frac{\xi_2-\xi_-}{\xi_1-\xi_-}},
\sqrt{\frac{\xi_2-\xi_1}{\xi_2-\xi_-}}\right]\right)^2}.
\end{equation}
\end{proposition}
%
%
Once the curvatures are known, we can even provide an explicit
expression for conformal geodesics, that is, we can write down and
integrate the equation of motion for $f$. The case $n=2$ is
trivial and $n=3$ already appears in \cite{musso}, whose method we
will follow closely. Therefore, we assume $n=4$ and $f$ to
$3$-generic, so that $\phi$ is given by \eqref{cartanformfrenet4}.
The key step is to provide a matrix $\Theta\in \mab(4)$, depending
on the conformal curvatures, such that the system
\eqref{eulquattro} defining the conformal geodesics is equivalent
to the differential equation
\begin{equation}\label{lax}
\dot\Theta = [\Theta, \phi\left(\tfrac{\di}{\di s}\right)].
\end{equation}
With some computation, we find that
\begin{equation}\label{lambda}
\Theta = \left(
\begin{array}{cccccc}
0  & 1  & -\mu_1-\mu_2^2 &\dot\mu_2 & \mu_2\mu_3 & 0 \\
0  & 0  & 0 & -\mu_2 & 0 & 1  \\
1  & 0  & 0 & 0 & 0 & -\mu_1-\mu_2^2 \\
0 & \mu_2 & 0 & 0 & 0 & \dot\mu_2 \\
0 & 0 & 0 & 0 & 0 & \mu_2\mu_3 \\
0 & 0 & 1 & 0 & 0 & 0
\end{array}
\right)
\end{equation}
satisfies \eqref{lax}. Therefore, if $e$ is the Frenet frame of
the conformal geodesic, $e\Theta e^{-1}$ does not depend on $s$
and defines a fixed element $\omega\in \mab(4)$. Substituting
\eqref{relazalgebrichecurv}, a straightforward calculation yields
the characteristic polynomial of $\Theta$:
\begin{equation}\label{chit}
\chi_\Theta(t) = t^6+2C_1t^4 -(1+C_3)t^2 -C_2^2,
\end{equation}
whose roots can be computed algebraically in dependence of
$C_1,C_2,C_3$. Note that, when $\mu_3\equiv 0$ (i.e. $C_2\equiv
0$), this coincides with $t$ times the characteristic polynomial
in \cite{musso}. Writing $e\Theta = \omega e$ by columns, we get
the system
\begin{equation}\label{sistpermoto}
\left\{\begin{array}{lcllcl} \omega e_0 & = & e_2; & \qquad \omega
e_5 & = & e_1 -
(\mu_1+\mu_2^2)e_2 + \dot\mu_2 e_3 + \mu_2\mu_3 e_4;\\[0.2cm]
\omega e_1 & = & e_0 + \mu_2 e_3; & \qquad \omega e_2& = & -(\mu_1+\mu_2^2)e_0 + e_5; \\[0.2cm]
\omega e_4& = & \mu_2\mu_3 e_0; &  \qquad \omega e_3 & = & \dot \mu_2 e_0 - \mu_2 e_1.
\end{array}\right.
\end{equation}
%
%
%
A repeated application of $\omega$ to the above system gives
\begin{equation*}
\begin{array}{lcl}
e_1 & = & \displaystyle -\frac{1}{\mu_2^2(\dot\mu_2^2+1)} \omega^5
e_0 - \frac{\dot \mu_2}{\mu_2(\dot \mu_2^2+1)} \omega^4 e_0 -
\frac{2(\mu_1+\mu_2^2)}{\mu_2^2(\dot\mu_2^2+1)} \omega^3e_0 \\[0.5cm]
& & \displaystyle -
\frac{2\dot\mu_2(\mu_1+\mu_2^2)}{\mu_2(\dot\mu_2^2+1)} \omega^2
e_0 + \frac{1+\dot\mu_2^2+\mu_2^2\mu_3^2}{\mu_2^2(\dot\mu_2^2+1)}
\omega e_0 + \frac{\dot \mu_2(1+\dot\mu_2^2
+\mu_2^2\mu_3^2)}{\mu_2(\dot\mu_2^2+1)}e_0.
\end{array}
\end{equation*}
Using \eqref{relazalgebrichecurv} and recalling that $\dot
e_0=e_1$ we are led to the equations of motion
\begin{equation}\label{moto}
\begin{array}{lcl}
\dot e_0 & = & \displaystyle -\frac{1}{\mu_2^2(\dot\mu_2^2+1)}
\omega^5 e_0 - \frac{\dot \mu_2}{\mu_2(\dot \mu_2^2+1)} \omega^4
e_0 +
\frac{\mu_2^2-2C_1}{\mu_2^2(\dot\mu_2^2+1)} \omega^3e_0 \\[0.5cm]
& & \displaystyle +
\frac{\dot\mu_2(\mu_2^2-2C_1)}{\mu_2(\dot\mu_2^2+1)} \omega^2 e_0
+ \frac{\mu_2^2+\dot\mu_2^2\mu_2^2+C_2^2}{\mu_2^4(\dot\mu_2^2+1)}
\omega e_0 + \frac{\dot \mu_2(\mu_2^2+\dot\mu_2^2\mu_2^2
+C_2^2)}{\mu_2^3(\dot\mu_2^2+1)}e_0.
\end{array}
\end{equation}
To solve \eqref{moto}, we study the endomorphism $M$ of $\erre^6$
represented, in the basis $\{\eta_0,\eta_A,\eta_{n+1}\}$, by the
matrix $\omega$. Observe that, from \eqref{Pt} and \eqref{chit},
\begin{equation*}
\chi_M(t)=\chi_\Theta(t)=P(-t^2)-t^2;
\end{equation*}
moreover, under the assumption \eqref{casiradici} $(iv)$, the
polynomial $P(-x)-x$ has three distinct real roots, $t_+$, $t_1$
and $t_2$, satisfying
\begin{equation}\label{radici}
t_2<-\xi_2<-\xi_1<t_1<0<-\xi_-<t_+,
\end{equation}
thus the eigenvalues of $M$ are
\begin{equation}\label{radici2}
\lambda=\sqrt{t_+},\quad -\lambda,\quad
i\tau_1=i\sqrt{\abs{t_1}},\quad -i\tau_1,\quad
i\tau_2=i\sqrt{\abs{t_2}},\quad -i\tau_2.
\end{equation}
and
\begin{equation}\label{decomponiamo}
\chi_\Theta(t) = \big(t^2-\lambda^2\big)\big(t^2+\tau_1^2\big)\big(t^2+\tau_2^2\big).
\end{equation}
The eigenvectors relative to the real eigenvalues $\pm \lambda$
can be computed via $\Theta$ and, once expressed in the moving
frame $\set{e_0,\ldots,e_5}$, they read
\begin{equation*}
{\phantom{\frac{a}{a}}}^t\pa{\lambda^2-\frac12\mu^2_2+C_1,
\frac{\pm \lambda-\mu_2\dot\mu_2}{\mu_2^2+\lambda^2},\pm
\lambda,\frac{\pm
\lambda\dot\mu_2+\mu_2}{\mu_2^2+\lambda^2},\frac{\pm
C_2}{\mu_2\lambda},1}.
\end{equation*}
One can use this explicit expression and easily check that they
are light-like. We call $S_1$ and $S_2$ the $1$-dimensional
eigenspaces relative to $\lambda$ and $-\lambda$ respectively.
Then we can decompose $\erre^6$ as $S_1\oplus S_2\oplus F$, where
$F$ is the orthogonal complement of $S_1\oplus S_2$. We observe
that $F$ is space-like and $M_{|F}$ is a skew-symmetric
endomorphism of $F$, with eigenvalues $\pm i\tau_1$ and $\pm
i\tau_2$. By standard linear algebra, $M_{|F}$ can therefore be
brought to the following block-diagonal form by an orthogonal
transformation:
\begin{equation*}
\left(
\begin{array}{cccc}
0  & -\tau_1  & 0 & 0 \\
\tau_1  & 0  & 0 & 0 \\
 0 & 0  & 0 & -\tau_2\\
 0 & 0 & \tau_2 & 0
\end{array}
\right).
\end{equation*}
Therefore, there exists an element $A\in\mob(4)$ such that
\begin{equation}\label{diag}
A\omega A^{-1}= \left(
\begin{array}{cccccc}
\lambda & 0 & 0 & 0 & 0 & 0\\
0 & 0 & -\tau_1 & 0 & 0 & 0\\
0 & \tau_1 & 0 & 0 & 0 & 0 \\
0 & 0 & 0 & 0 & -\tau_2 & 0\\
0 & 0 & 0 & \tau_2 & 0 & 0\\
0 & 0 & 0 & 0 & 0 & -\lambda
\end{array}
\right)
\end{equation}
Since we are interested in solving equations \eqref{moto} up to a conformal transformation of $\mathds{Q}_4$,
we can assume that $\omega$ has the form at the RHS of \eqref{diag} from the start, possibly substituting $e$
with $Ae$.\\
Now we set $h(s)=B^{-1}e_0(s)$, with
\begin{equation*}
B= \left(
\begin{array}{cccccc}
 1 & 0 & 0 & 0 & 0 & 0\\
 0 & i & -i & 0 & 0 & 0\\
 0 & 1 & 1 & 0 & 0 & 0\\
 0 & 0 & 0 & i & -i & 0\\
0 & 0 & 0 & 1 & 1 & 0\\
0 & 0 & 0 & 0 & 0 & 1
\end{array}
\right),
\end{equation*}
so that $B^{-1}\omega B$ is diagonal.
Setting
\begin{align*}
&a_0=\frac{\dot \mu_2(\mu_2^2+\dot\mu_2^2\mu_2^2 +C_2^2)}{\mu_2^3(\dot\mu_2^2+1)},
\quad a_1=\frac{\mu_2^2+\dot\mu_2^2\mu_2^2+C_2^2}{\mu_2^4(\dot\mu_2^2+1)},\quad
a_2= \frac{\dot\mu_2(\mu_2^2-2C_1)}{\mu_2(\dot\mu_2^2+1)} \\
&a_3=\frac{\mu_2^2-2C_1}{\mu_2^2(\dot\mu_2^2+1)}, \quad a_4=-\frac{\dot \mu_2}{\mu_2(\dot \mu_2^2+1)} ,\quad
a_5=-\frac{1}{\mu_2^2(\dot\mu_2^2+1)}
\end{align*}
and
$$
\begin{array}{ll}
I(t)=a_4t^4-a_2t^2+a_0, & \quad J(t)=a_5t^5-a_3t^3+a_1t \\[0.2cm]
\tilde I(t)=a_4t^4+a_2t^2+a_0, & \quad \tilde J(t)=a_5t^5+a_3t^3+a_1t,
\end{array}
$$
the equations of motion become
$$
\begin{array}{ll}
\dot h^0 = \sq{\tilde I(\lambda)+\tilde J(\lambda)}h^0; & \quad
\dot h^1 = \sq{I(\tau_1)+iJ(\tau_1)}h^1; \\[0.2cm]
\dot h^2 = \sq{I(\tau_1)-iJ(\tau_1)}h^2; & \quad \dot h^3 =
\sq{I(\tau_2)+iJ(\tau_2)}h^3; \\[0.2cm]
\dot h^4 = \sq{I(\tau_2)-iJ(\tau_2)}h^4; & \quad \dot h^5 =
\sq{\tilde I(\lambda)-\tilde J(\lambda)}h^5;
\end{array}
$$
using \eqref{energy}, \eqref{chit} and \eqref{decomponiamo} we
find that
$$
\begin{array}{ll}
\tilde I(\lambda)=-\dfrac{\lambda^2\dot\mu_2}{\mu_2(\mu_2^2+\lambda^2)}+\dfrac{\dot\mu_2}{\mu_2}; & \qquad
\tilde J(\lambda)=\dfrac{\lambda}{\mu_2^2+\lambda^2};\\[0.4cm]
I(\tau_i)=-\dfrac{\tau_i^2\dot\mu_2}{\mu_2(\mu_2^2-\tau_i^2)}+\dfrac{\dot\mu_2}{\mu_2}, & \qquad
J(\tau_i)=\dfrac{\tau_i}{\mu_2^2-\tau_i^2}, \qquad i=1,2.
\end{array}
$$
Integrating the above equalities and observing that, by
\eqref{radici}, \eqref{radici2}, \eqref{decomponiamo} and $\mu_2^2
\in [\xi_1,\xi_2]$ we must have $\mu^2_2\in(\tau_1^2,\tau_2^2)$,
we are led to the solutions
\begin{align*}
e_0^0=&\ p_0\sqrt{\mu_2^2+\lambda^2}\exp\pa{\lambda\int_{s_0}^s\frac{\di t}{\mu_2^2+\lambda^2}};\\
e_0^1=&\ \rho_1 \sqrt{\mu_2^2-\tau_1^2}\sin \pa{\tau_1\int_{s_0}^s\frac{\di t}{\tau_1^2-\mu_2^2}-\theta_1};\\
e_0^2=&\ \rho_1 \sqrt{\mu_2^2-\tau_1^2}\cos \pa{\tau_1\int_{s_0}^s\frac{\di t}{\tau_1^2-\mu_2^2}-\theta_1};\\
e_0^3=&\ \rho_2 \sqrt{\tau_2^2-\mu_2^2}\sin \pa{\tau_2\int_{s_0}^s\frac{\di t}{\tau_2^2-\mu_2^2}-\theta_2};\\
e_0^4=&\ \rho_2 \sqrt{\tau_2^2-\mu_2^2}\cos \pa{\tau_2\int_{s_0}^s\frac{\di t}{\tau_2^2-\mu_2^2}-\theta_2};\\
e_0^5=&\ p_5\sqrt{\mu_2^2+\lambda^2}\exp\pa{-\lambda\int_{s_0}^s\frac{\di t}{\mu_2^2+\lambda^2}},
\end{align*}
for arbitrary constants
$p_0,p_5,\rho_1,\rho_2,\theta_1,\theta_2\in
\erre$, $\rho_i\ge0$.\\
The projectivization of these solutions represents a conformal
geodesic if and only if $e_0$ is light-like, and this happens when
the constants satisfy the conditions $2p_0p_5=\rho_1^2-\rho_2^2$
and $2p_0p_5\lambda^2=\rho_2^2\tau_2^2-\rho_1^2\tau_1^2$. The most
general light-like solution is therefore
\begin{align*}
e_0^0=&\ A\frac{\sqrt{\tau_2^2-\tau_1^2}}{\sqrt2}\sqrt{\mu_2^2+\lambda^2}\ \exp\pa{\lambda\int_{s_0}^s\frac{\di t}{\mu_2^2+\lambda^2}};\\
e_0^1=&\ \sqrt{\lambda^2+\tau_2^2}\sqrt{\mu_2^2-\tau_1^2}\sin \pa{\tau_1\int_{s_0}^s\frac{\di t}{\tau_1^2-\mu_2^2}-\theta_1};\\
e_0^2=&\ \sqrt{\lambda^2+\tau_2^2} \sqrt{\mu_2^2-\tau_1^2}\cos \pa{\tau_1\int_{s_0}^s\frac{\di t}{\tau_1^2-\mu_2^2}-\theta_1};\\
e_0^3=&\ \sqrt{\lambda^2+\tau_1^2} \sqrt{\tau_2^2-\mu_2^2}\sin \pa{\tau_2\int_{s_0}^s\frac{\di t}{\tau_2^2-\mu_2^2}-\theta_2};\\
e_0^4=&\ \sqrt{\lambda^2+\tau_1^2} \sqrt{\tau_2^2-\mu_2^2}\cos \pa{\tau_2\int_{s_0}^s\frac{\di t}{\tau_2^2-\mu_2^2}-\theta_2};\\
e_0^5=&\ \frac1A\frac{\sqrt{\tau_2^2-\tau_1^2}}{\sqrt2}\sqrt{\mu_2^2+\lambda^2}\ \exp\pa{-\lambda\int_{s_0}^s\frac{\di t}{\mu_2^2+\lambda^2}}.
\end{align*}
Let us denote by $\tilde e_0$ the particular solution with $A=1$,
$\theta_1=\theta_2=0$. Then the general solution $e_0$ is obtained
as $e_0=M\tilde e_0$, where
\begin{equation*}
M=\left(
\begin{array}{cccccc}
A & 0 & 0 & 0 & 0 & 0\\
0 & \cos\theta_1 & -\sin\theta_1 & 0 & 0 & 0\\
0 & \sin\theta_1 & \cos\theta_1 & 0 & 0 & 0 \\
0 & 0 & 0 &  \cos\theta_2 & -\sin\theta_2 & 0\\
0 & 0 & 0 & \sin\theta_2 & \cos\theta_2 & 0\\
0 & 0 & 0 & 0 & 0 & A^{-1}
\end{array}
\right)\in \mob(4).
\end{equation*}
Therefore, up to a conformal motion of $\mathds{Q}_4$, the
conformal geodesics are given by
\begin{align*}
e_0^0=&\ \frac{\sqrt{\tau_2^2-\tau_1^2}}{\sqrt2}\sqrt{\mu_2^2+\lambda^2}\ \exp\pa{\lambda\int_{s_0}^s\frac{\di t}{\mu_2^2+\lambda^2}};\\
e_0^1=&\ \sqrt{\lambda^2+\tau_2^2}\sqrt{\mu_2^2-\tau_1^2}\sin \pa{\tau_1\int_{s_0}^s\frac{\di t}{\tau_1^2-\mu_2^2}};\\
e_0^2=&\ \sqrt{\lambda^2+\tau_2^2} \sqrt{\mu_2^2-\tau_1^2}\cos \pa{\tau_1\int_{s_0}^s\frac{\di t}{\tau_1^2-\mu_2^2}};\\
e_0^3=&\ \sqrt{\lambda^2+\tau_1^2} \sqrt{\tau_2^2-\mu_2^2}\sin \pa{\tau_2\int_{s_0}^s\frac{\di t}{\tau_2^2-\mu_2^2}};\\
e_0^4=&\ \sqrt{\lambda^2+\tau_1^2} \sqrt{\tau_2^2-\mu_2^2}\cos \pa{\tau_2\int_{s_0}^s\frac{\di t}{\tau_2^2-\mu_2^2}};\\
e_0^5=&\ \frac{\sqrt{\tau_2^2-\tau_1^2}}{\sqrt2}\sqrt{\mu_2^2+\lambda^2}\ \exp\pa{-\lambda\int_{s_0}^s\frac{\di t}{\mu_2^2+\lambda^2}}.
\end{align*}
\section{Appendix}
In this section we provide a somewhat detailed proof of
Proposition \ref{existvariation}. We refer to Section
\ref{sezionearco} for notations.
\begin{proposition*}
For every $p\in I$ there exists an open neighbourhood $U$ of $p$
such that the following holds: for every collection of $(n-1)$
smooth functions $\lambda^\alpha \in C^\infty(I)$ with compact
support $C$ included in $U$, there exist $\eps$ sufficiently
small, a variation $v: I \times (-\eps,\eps)\ra \mathds{Q}_n$ and
a special third order frame $e: U \times (-\eps,\eps)\ra \mob(n)$
along $v$ such that $\lambda^\alpha_0(p,0)= \lambda^\alpha(p)$ for
every $p\in U$.
\end{proposition*}
\begin{proof}
Consider the immersion $f:I\ra \mathds{Q}_n$. We fix local
coordinates $x$ on some neighbourhood $U_0$ of $p$ and
$(w,y^\alpha)$ on a neighbourhood $V_0\supset f(U_0)$ of $f(p)\in
\mathds{Q}_n$, with the property that the local expression of $f$
is
\begin{equation*}
f:x \longmapsto(x,0).
\end{equation*}
Any variation $v$ can be locally expressed, at least if $t$ is
small, as
\begin{equation}\label{varlocale}
v(x,t)=(x,z^\alpha(x,t)),
\end{equation}
where $z^\alpha$ are real-valued functions. Note that $f_t$ is an
immersion for every $t$.
Locally around $f(p)\in \mathds{Q}_n$, both $\set{\di w,\di
y^\alpha}$ and $\set{\psi^A_0}$ (with respect to a section
$\sigma$ of $\pi$) are local bases of the space of $1$-forms, so
there exists a non singular matrix $B$ such that
\begin{equation*}
\psi^A_0=B^A_1\di w+B^A_\alpha \di y^\alpha.
\end{equation*}
Under a change of sections of $\mob(n)\ra \mathds{Q}_n$
\begin{equation*}
\tilde\psi^A_0=r^{-1}A^B_A\psi^B_0,
\end{equation*}
so
\begin{equation*}
\tilde\psi^A_0=\tilde B^A_1\di w+\tilde B^A_\alpha \di y^\alpha,
\end{equation*}
with
\begin{equation*}
\tilde B^A_C=r^{-1}A^B_AB^B_C.
\end{equation*}
Therefore we can always choose a section $\sigma$, defined on some
neighbourhood $V_1\subset V_0$ of $f(p)$, such that
$B^1_\alpha=0$, namely the span of $\psi^1_0$ coincides with that
of the form $\di w$. Since $B$ is nonsingular, this implies that
both $B^1_1\neq 0$ and $(B^\alpha_\beta)$ is nonsingular. The
subgroup of $\mob(n)$ preserving such frames is constituted by the
matrices with $A^\alpha_1=0$. Since $A\in SO(n)$, this implies
$A^1_\alpha=0$, $A^1_1=1$ and the transformation laws for the
matrix $B$ are
\begin{equation*}
\tilde B^1_1=r^{-1}B^1_1,\qquad\tilde B^\alpha_1=r^{-1}A^\beta_\alpha B^\beta_1,\qquad\tilde B^\alpha_\beta=r^{-1}A^\gamma_\alpha B^\gamma_\beta.
\end{equation*}
Now the expression of $\psi^A_0$ is
\begin{equation}\label{changebases}
\psi^1_0=B^1_1\di w,\qquad \psi^\alpha_0=B^\alpha_1\di w+B^\alpha_\beta \di y^\beta.
\end{equation}
Define $U= U_0 \cap f^{-1}(V_0)$, and choose an arbitrary
collection of smooth functions $\{\lambda^\alpha\}$ supported in
some compact set $C\subset U$. Let $v$ be a variation of the form
\eqref{varlocale} on $U$. Pulling back \eqref{changebases}:
\begin{equation}\label{phiecarta}
\phi^1_0=\big(B^1_1\circ v\big) \di x,\qquad \phi^\alpha_0=\big(B^\alpha_1\circ v\big) \di x+\big(B^\alpha_\beta\circ v\big)
\pa{\frac{\partial z^\beta}{\partial x}\di x+\frac{\partial z^\beta}{\partial t}\di t}
\end{equation}
Observe that, since $B^1_1\neq 0$ pointwise, the first equation in
\eqref{phiecarta} implies that $\{\phi^1_0\}$ never vanishes. But,
on the other hand, $\phi^A_0=\phi^A_0(t)+\lambda^A_0 \di t$ so
\begin{equation}\label{sistema}
\begin{cases}
\lambda^1_0 = 0; \qquad \phi^1_0(t)= \big(B^1_1\circ v\big) \di x\\[0.2cm]
\big(B^\alpha_1\circ v\big)\di x+\big(B^\alpha_\beta\circ v\big) \dfrac{\partial z^\beta}{\partial x}\di x=\phi^\alpha_0(t);\\[0.2cm]
\big(B^\alpha_\beta\circ v\big) \dfrac{\partial z^\beta}{\partial t}=\lambda^\alpha_0.
\end{cases}
\end{equation}
Restricting the third equality to $t=0$ we get
\begin{equation}\label{datoin}
\big(B^\alpha_\beta\circ f\big) \frac{\partial z^\beta}{\partial t}\left.\!\!\!\frac{}{}\right|_{t=0}=\lambda^\alpha_0(0).
\end{equation}
Since $(B^\alpha_\beta)$ is always nonsingular, we can define
\begin{equation*}
z^\alpha(x,t)=t \big(B^{-1}\circ f\big)^\alpha_\beta\lambda^\beta,
\end{equation*}
and we observe that, by \eqref{datoin},
$\lambda^\alpha_0(0)=\lambda^\alpha$, $z^\alpha(x,0)=0$ and
$z^\alpha(x,t)=0$ for every $t\in (-\eps,\eps)$ and $x\in
U\backslash C$. The second condition allows the variation $v$
defined in \eqref{varlocale} to be extended smoothly to the whole
$I\times (-\eps,\eps)$ by setting $v(p,t)=f(p)$ for $p\not\in U$.
Moreover, $e$ is a local zeroth order frame along $v$ defined on
the whole $U$ and the third of \eqref{sistema} is satisfied with
$\lambda^\alpha_0(x,t)$ defined by
\begin{equation*}
\lambda^\alpha_0(x,t)=\pa{B\circ v}^\alpha_\beta\pa{B^{-1}\circ f}^\beta_\gamma\lambda^\gamma(x).
\end{equation*}
Now we need to perform the frame reduction for frames along
variations. The procedure is almost the same as without the
dependence on $t$, but paying attention that, step by step, the
neighbourhood $U$ on which the frame is defined be kept fixed.
Under a generic change of zeroth order frames along $v$,
$\phi^A_0$ change according to
\begin{equation}\label{cambiofi}
\tilde\phi^1_0=r^{-1}\pa{R^1_1\phi^1_0+R^\beta_1\phi^\beta_0},\qquad\tilde\phi^\alpha_0=r^{-1}\pa{R^1_\alpha\phi^1_0+
R^\beta_\alpha\phi^\beta_0}
\end{equation}
where $R\in SO(n)$. Decomposing
$\phi^A_0=\phi^A_0(t)+\lambda^A_0\di t$ and observing that
$\lambda^1_0=0$ by \eqref{sistema}, we get
\begin{equation}\label{cambiooo}
\begin{array}{ll}
\tilde\lambda^1_0=r^{-1}R^\alpha_1\lambda^\alpha_0, & \tilde\lambda^\alpha_0=r^{-1}R^\beta_\alpha\lambda^\beta_0. \\[0.2cm]
\tilde\phi^1_0(t)=r^{-1}\pa{R^1_1\phi^1_0(t)+R^\beta_1\phi^\beta_0(t)}, & \tilde\phi^\alpha_0(t)=
r^{-1}\pa{R^1_\alpha\phi^1_0(t)+R^\beta_\alpha\phi^\beta_0(t)}
\end{array}
\end{equation}
Since the set $\{\phi^1_0(t),\phi^\beta_0(t)\}$ has rank $1$ on
$U$ being $f_t$ an immersion, we can choose a suitable $R$ so that
$\widetilde{\phi}^\alpha_0(t)=0$, that is, $e_t$ is a first order
frame for every $t$. $R$ is globally defined on $U\times
(-\eps,\eps)$, as it is apparent from the linear algebra
procedure involved in its construction.\\
Observe that the matrix $(R^\alpha_\beta)$ is nonsingular for
every $(x,t)$; indeed, let $v^\alpha$ be such that $R^\beta_\alpha
v^\alpha=0$, then, since $\tilde\phi^\alpha_0(t)=0$, using
\eqref{cambiofi} we get
\begin{equation*}
0=r^{-1}v^\alpha\pa{R^1_\alpha\phi^1_0(t)+R^\beta_\alpha\phi^\beta_0(t)}=r^{-1}v^\alpha R^1_\alpha\phi^1_0(t),
\end{equation*}
therefore
\begin{equation*}
v^\alpha R^1_\alpha \phi^1_0(t)=0.
\end{equation*}
Since $\phi^1_0(t)$ does not vanish by \eqref{sistema}, we deduce
that $R^1_\alpha v^\alpha=0$, that is  $R^A_\alpha v^\alpha=0$,
which implies $v^\alpha=0$ by the invertibility of $R$. Therefore
$(R^\alpha_\beta)$ is nonsingular and in particular
$(\tilde\lambda^\alpha_0)=0$ if and only if
$(\lambda^\alpha_0)=0$. Since the action of the group of rotations
and homotheties is transitive on $\erre^{n-1}\backslash \{0\}$, by
\eqref{cambiooo} we can perform a change of first order frames by
means of a suitable globally defined matrix of the kind
\begin{equation*}
\tilde R=
\left(
\begin{array}{cc}
 1 & 0 \\
 0 & C \\
\end{array}
\right), \qquad C\in SO(n-1),
\end{equation*}
to smoothly restore $\tilde\lambda^\alpha_0(x,0)$ to its original
value $\lambda^\alpha(x)$. We now proceed with a frame reduction
to the second order keeping $\lambda^\alpha(x,0)$ unaltered. Write
\begin{equation}\label{framenor}
{\phi^\alpha_1}_{(p,t)}=h^\alpha(p,t)\phi^1_0(t)_p+\Lambda^\alpha_1(p,t)\di t=h^\alpha(p,t){\phi^1_0}_{(p,t)}
+\lambda^\alpha_1(p,t)\di t,
\end{equation}
where $\lambda^\alpha_1=-h^\alpha\lambda^1_0+\Lambda^\alpha_1$ are smooth functions satisfying $\lambda^\alpha_1(p,t)=0$ for $p\in U\setminus C$ and $t\in(-\eps,\eps)$.\\
A change of first order frame has values in
\begin{equation}
  \mob(n)_1=\set{\left(\begin{array}{cccc}
  r^{-1}&x &{}^tyB&\frac{1}{2}r\pa{x^2+|y|^2}\\0&1&0&rx\\0&0&B&ry\\0&0&0&r
  \end{array}\right)\left| \begin{array}{l}
    r\in\erre^+,\\ B\in SO(n-1),\\x\in\erre,y\in\erre^{n-1}
  \end{array}\right.},
\end{equation}
Under a change of first order frame, the coefficients $h^\alpha$
changes according to \eqref{cambhalfamono}: $\tilde h^\alpha
=rB^\beta_\alpha\pa{h^\beta-y^\beta}$. Considering the globally
defined frame $e=eK$, where $K$ has the form above with
$B=I_{n-1}$ and $y^\alpha=h^\alpha$ we can pass to a second order
frame on $U$.
Moreover, by the change of $\lambda^\alpha_0$ in \eqref{cambiooo}
we get
$\widetilde{\lambda}^\alpha_0(x,0)=\lambda^\alpha_0(x,0)=\lambda^\alpha(x)$.
The last two steps involve the coefficients $p^\alpha$ defined by
$\phi^0_\alpha(t)=p^\alpha\phi^1_0(t)$, and on $\phi^0_0(t)$.
Since the curve is $1$-generic, up to choosing $\eps$ small enough
we can assume that every curve of the variation is $1$-generic.
The procedure is then identical to the one in Section
\ref{sez_frenet}, and arguing as above it is easy to find the
desired special third order frame globally defined on $U$. This
concludes the proof.
\end{proof}
\vspace{1cm}
\textbf{Acknowledgement:} the authors express their gratitude to
professor L. Pizzocchero for his kind helpfulness and valuable
suggestions.
\bibliographystyle{amsplain}
\bibliography{BiblioConformecurve}

\end{document}